\documentclass[a4paper,reqno,10pt]{amsart}

\usepackage{epsf,graphicx,epsfig}
\usepackage{amscd}
\usepackage{amsmath,latexsym,amssymb,amsthm}
\usepackage[nospace,noadjust]{cite}
\usepackage{textcomp}
\usepackage{setspace,cite}
\usepackage{lscape,fancyhdr,fancybox}
\usepackage{stmaryrd}
\usepackage[all,cmtip]{xy}
\usepackage{tikz}
\usepackage{cancel}
\usetikzlibrary{shapes,arrows,decorations.markings}
\newcommand{\verteq}{\rotatebox{90}{$\,=$}}
\newcommand{\equalto}[2]{\underset{\scriptstyle\overset{\mkern4mu\verteq}{#2}}{#1}}

\usepackage{graphicx}

\usepackage{color}
\usepackage{url}
\usepackage{enumerate}
\usepackage[mathscr]{euscript}

  \theoremstyle{plain}

\swapnumbers
    \newtheorem{thm}{Theorem}[section]
    \newtheorem{prop}[thm]{Proposition}

    \newtheorem{subsec}[thm]{}
\theoremstyle{definition}
    \newtheorem{defn}[thm]{Definition}

    \newtheorem{exam}[thm]{Example}

\theoremstyle{remark}

\title{}
\author{}
\date{}
\usepackage{amssymb}

\usepackage{hyperref}
\hypersetup{
	colorlinks,
	citecolor=blue,
	filecolor=black,
	linkcolor=blue,
	urlcolor=black
}

\begin{document}

\title{Embedding tensors on Hom-Lie algebras}

\begin{abstract}
The notion of embedding tensors and the associated tensor hierarchies form an effective tool for the construction of supergravity and higher gauge theories. Embedding tensors and related structures are extensively studied also in the mathematics literature. On the other hand, Hom-Lie algebras were introduced in the study of $q$-deformations of Witt and Virasoro algebras. In this paper, we first introduce embedding tensors on a Hom-Lie algebra with respect to a given representation. An embedding tensor naturally induces a Hom-Leibniz algebra structure. We construct a graded Lie algebra that characterizes embedding tensors as its Maurer-Cartan elements. Using this, we define the cohomology of an embedding tensor and realize it as the cohomology of the induced Hom-Leibniz algebra with coefficients in a suitable representation. A triple consisting of a Hom-Lie algebra, a representation and an embedding tensor is called a Hom-Lie-Leibniz triple. We construct the controlling $L_\infty$-algebra of a given Hom-Lie-Leibniz triple. Next, we define the cohomology of a Hom-Lie-Leibniz triple that governs the deformations of the structure. Finally, we introduce homotopy embedding tensors, $HLeib_\infty$-algebras and find their relations.
\end{abstract}

\author{Apurba Das}
\address{Department of Mathematics, Indian Institute of Technology, Kharagpur 721302, West Bengal, India.}
\email{apurbadas348@gmail.com}

\author{Abdenacer Makhlouf}
\address{University of Haute Alsace, IRIMAS - D\'{e}partement de Math\'{e}matiques, 18, rue des fr\`{e}res Lumi\`{e}re, F-68093 Mulhouse, France}
\email{abdenacer.makhlouf@uha.fr}

\maketitle



\medskip



\tableofcontents

\section{Introduction}
The notion of `embedding tensor' and the associated tensor hierarchies provide a useful tool for the construction of supergravity theories and further higher gauge theories. Given a Lie algebra $(\mathfrak{g}, [~,~])$ and a representation $(V, \rho)$, a linear map $T: V \rightarrow \mathfrak{g}$ is said to be an embedding tensor on $\mathfrak{g}$ with respect to the representation $(V, \rho)$ if the map $T$ satisfies the following identity:
\begin{align}
    [T(u), T(v)] = T (\rho (Tu) v), \text{ for } u, v \in V.
\end{align}
A triple consisting of a Lie algebra, a representation and an embedding tensor is called a Lie-Leibniz triple. See \cite{bon,bon2,wit,wit2,wit3,wit4,hohm,kotov-strobl,nicolai} and references therein for extensive (physics) literature about embedding tensors, Lie-Leibniz triples and associated tensor hierarchies. Recently, this topic got the attention of many mathematical physicists and mathematicians working on higher structures \cite{kotov-strobl,lavau,lavau-p,lavau-stas,sheng-embd,strobl-wage}. Among others, Kotov and Strobl first observe the mathematical nature behind the calculations appearing from embedding tensors and their tensor hierarchies in the physics literature \cite{kotov-strobl}. They showed that an embedding tensor naturally gives rise to a Leibniz algebra that further gives rise to an $L_\infty$-algebra structure. Thus, embedding tensors are related to rich mathematical structures. In the physical application, they observed that one could focus on Leibniz algebras (instead of embedding tensors) in the construction of the corresponding gauge theories. In \cite{sheng-embd} Sheng, Tang and Zhu added some more mathematical flavours to this subject. Namely, they first constructed the controlling structures for embedding tensors and Lie-Leibniz triples. As applications, they considered cohomology, deformations and homotopy theory for embedding tensors and Lie-Leibniz triples. Their idea is similar to the cohomology and deformation theory for Rota-Baxter algebras \cite{laza,das-mishra}.

\medskip

On the other hand, the notion of Hom-Lie algebras was introduced by Hartwig, Larsson and Silvestrov in the study of $q$-deformations (also called quantum deformations) of the Witt and Virasoro algebra \cite{hls}. A Hom-Lie algebra is a vector space equipped with a skew-symmetric bilinear bracket whose Jacobi identity is twisted by a bracket preserving linear homomorphism (called the Hom-Jacobi identity). We often call a vector space together with a linear homomorphism by a Hom-vector space. Morphisms between Hom-vector spaces are those morphisms between vector spaces that preserve the corresponding linear homomorphisms. With this, a Hom-Lie algebra can be realized as a Hom-vector space with a skew-symmetric bilinear bracket satisfying the Hom-Jacobi identity (see Definition \ref{defn-hom-lie}). When the linear homomorphism of the underlying Hom-vector space is the identity map, one gets the usual Lie algebra. In the last fifteen years, many classical studies on Lie algebras have been generalized to Hom-Lie algebras. For instance, in \cite{sheng-rep} Sheng considered representations of Hom-Lie algebras, see also \cite{benayadi}. Subsequently, cohomology and deformation theory of Hom-Lie algebras are extensively studied in \cite{makhlouf1,arfa,makhlouf2}. The notion of $HL_\infty$-algebras (also called strongly homotopy Hom-Lie algebras) was introduced and studied by Sheng and Chen as the homotopy version of Hom-Lie algebras \cite{sheng-hl}. Note that $HL_\infty$-algebras can be realized as a twisted version of $L_\infty$-algebras introduced by Lada and Stasheff \cite{lada-s} (see also \cite{lada-markl}). Other algebraic structures (e.g. associative algebras, bialgebras, Leibniz algebras, pre-Lie algebras) twisted by linear homomorphisms are also widely studied in the literature \cite{arfa,liu-song,makhlouf0,makhlouf3,makhlouf1,makhlouf2,das,larsson-sil,sheng-hl,wang,mishra-naolekar}.

\medskip

Our main aim in this paper is to extend the mathematical study of embedding tensors to the context of Hom-Lie algebras. More precisely, we start by defining embedding tensors on a Hom-Lie algebra with respect to a representation. We find some nice examples of embedding tensors (some of which are new even for Lie algebras). We show that a Hom-Lie algebra together with a representation gives rise to a Hom-Leibniz algebra structure (the hemi-semidirect product). It turns out that the graph of an embedding tensor is a Hom-Leibniz subalgebra of the hemi-semidirect product. As a consequence, we show that an embedding tensor $T$ naturally induces a Hom-Leibniz algebra structure. Next, given a Hom-Lie algebra and a representation, we construct a new graded Lie algebra which we call the `derived graded Lie algebra'. This graded Lie algebra is obtained as a byproduct of Balavoine's graded Lie bracket (that characterizes Hom-Leibniz algebras as Maurer-Cartan elements) and Kosmann-Schwarzbach's derived bracket. We show that an embedding tensor on a Hom-Lie algebra with respect to a representation can be characterized by Maurer-Cartan elements in the corresponding derived graded Lie algebra. As a consequence of this characterization, we can define the cohomology of a given embedding tensor $T.$ We show that the cohomology of the embedding tensor $T$ can be realized as the cohomology of the induced Hom-Leibniz algebra with coefficients in a suitable representation. Moreover, there is a homomorphism from the cohomology of an embedding tensor $T$ to the cohomology of the induced Hom-Leibniz algebra with coefficients in the adjoint representation. Finally, the first cohomology group of the embedding tensor $T$ classifies the set of all equivalence classes of infinitesimal deformations of $T$ (keeping the underlying Hom-Lie algebra and representation intact).

\medskip

In the next part, we mainly deal with the full triple that consists of a Hom-Lie algebra, a representation and an embedding tensor. We call such a triple as Hom-Lie-Leibniz triple. We construct an explicit $L_\infty$-algebra whose Maurer-Cartan elements correspond to Hom-Lie-Leibniz triples. Here the $L_\infty$-algebra is constructed as a byproduct of Balavoine's graded Lie bracket and Voronov's higher derived brackets. Subsequently, we study the cohomology of a Hom-Lie-Leibniz triple (with coefficients in the adjoint representation). We also define representations of a Hom-Lie-Leibniz triple and introduce cohomology with coefficients in any arbitrary representation. We show that our cohomology is helpful to classify infinitesimal deformations of a Hom-Lie-Leibniz triple.

\medskip

In the final part of this paper, we mainly focus on homotopy embedding tensors on a $HL_\infty$-algebra with respect to a representation. Given a $HL_\infty$-algebra and a representation of it, we first construct a new $L_\infty$-algebra. This $L_\infty$-algebra
is the homotopy generalization of the derived graded Lie algebra constructed 
earlier. For this reason, we call this $L_\infty$-algebra as the derived $L_\infty$-algebra. A homotopy embedding tensor is defined by a Maurer-Cartan element of this derived $L_\infty$-algebra. We have observed earlier that an embedding tensor (on a Hom-Lie algebra with respect to a representation) induces a Hom-Leibniz algebra structure. To generalize this result in the homotopy context, we first introduce $HLeib_\infty$-algebras (strongly homotopy Hom-Leibniz algebras). Finally, we show that a homotopy embedding tensor naturally induces a $HLeib_\infty$-algebra structure.

\medskip

We organize the paper as follows. In Section \ref{sec2}, we recall some necessary background on Hom-Lie algebras and Hom-Leibniz algebras. In Section \ref{sec3}, we define and study some basic properties of embedding tensors. In the same section, we also study Maurer-Cartan characterization and cohomology theory of embedding tensors. Hom-Lie-Leibniz triples are considered in Section \ref{sec5} and their Maurer-Cartan characterization is also given. In Section \ref{sec6}, we introduce the cohomology of a Hom-Lie-Leibniz triple and study deformation theory. Finally, homotopy embedding tensors, $HLeib_\infty$-algebras are introduced in Section \ref{sec7} and their relations are discussed.

\medskip

All (graded) vector spaces, linear and multilinear maps,
unadorned tensor products and wedge products are over a field ${\bf k}$ of characteristic $0$.

\medskip

\section{Some background on Hom-Lie algebras and Hom-Leibniz algebras}\label{sec2}
In this section, we recall some basic definitions of Hom-Lie algebras and Hom-Leibniz algebras. In particular, we recall the graded Lie algebra (a generalization of Balavoine's graded Lie algebra) whose Maurer-Cartan elements correspond to Hom-Leibniz algebras. This graded Lie algebra plays a crucial role in our whole study. Our main references are \cite{hls,sheng-rep,makhlouf1,makhlouf3,mukh-saha,loday-pira}.

\begin{defn}\label{defn-hom-lie}
A {\bf Hom-Lie algebra} is a triple $(\mathfrak{g}, [~,~], \alpha)$ consisting of a vector space $\mathfrak{g}$ with a skew-symmetric bilinear bracket $[~,~] : \mathfrak{g} \times \mathfrak{g} \rightarrow \mathfrak{g}$, $(x, y) \mapsto [x,y]$ and a linear map $\alpha : \mathfrak{g} \rightarrow \mathfrak{g}$ (called the twist) satisfying
\begin{align}
&\alpha ([x,y]) = [\alpha (x) , \alpha (y)], \label{mul-condn}\\
[\alpha (x) , [y, z & ]] + [\alpha (y) , [z, x]] + [\alpha (z), [x, y]] = 0, \label{hom-jac}
\end{align}
for $x, y, z \in \mathfrak{g}$. We often denote a Hom-Lie algebra as above by $\mathfrak{g}^{[~,~]}_\alpha$.
\end{defn}

The condition (\ref{mul-condn}) is called the `multiplicativity' and the condition (\ref{hom-jac}) is called the `Hom-Jacobi identity'. For this reason, a Hom-Lie algebra as of Definition \ref{defn-hom-lie} is often called `multiplicative' \cite{makhlouf1,makhlouf3}. A morphism between two Hom-Lie algebras is given by a linear map that preserves the brackets and commutes with the twist maps. 

\begin{exam}
Let $(\mathfrak{g}, [~,~])$ be a Lie algebra and $\alpha : \mathfrak{g} \rightarrow \mathfrak{g}$ be a Lie algebra homomorphism. Then $\mathfrak{g}^{\alpha \circ [~,~]}_\alpha = (\mathfrak{g} , \alpha \circ [~,~], \alpha)$ is a Hom-Lie algebra. In fact, $\mathfrak{g}^{\alpha^n \circ [~,~]}_{\alpha^n}$ is a Hom-Lie algebra, for any $n \geq 1$.
\end{exam}

\begin{exam}
    Let $\mathfrak{g}$ be the $4$-dimensional vector space spanned by the set $\{ e_1, e_2, e_3, e_4 \}$. We define a bilinear skewsymmetric bracket $[~,~] : \mathfrak{g} \times \mathfrak{g} \rightarrow \mathfrak{g}$ and a linear map $\alpha : \mathfrak{g} \rightarrow \mathfrak{g}$ by
    \begin{align*}
        [e_1, e_2] = - ae_3,& ~~~~ [e_1, e_3] = b e_2, ~~~~ [e_1, e_4] = 0, ~~~~ [e_2, e_3 ]= 0, ~~~~ [e_2, e_4] = - a e_2, ~~~~ [e_3, e_4] = a e_3, \\
    &    \alpha (e_1) = - e_1, ~~~~ \alpha (e_2) = e_2, ~~~~ \alpha (e_3) = - e_3 ~~~ \text{ and } ~~~ \alpha (e_4) = e_4,
    \end{align*}
    where $a, b$ are any scalars. Then it is easy to verify that $\mathfrak{g}^{ [~,~]}_\alpha$ is a Hom-Lie algebra.
\end{exam}

\begin{exam}
    Let $(A, \mu, \alpha)$ be a Hom-associative algebra \cite{makhlouf0}. That is, $A$ is a vector space equipped with a bilinear map $\mu : A \times A \rightarrow A,~ (a, b) \mapsto a \cdot b$ and a linear map $\alpha : A \rightarrow A$ that satisfies
    \begin{align*}
        \alpha (a \cdot b) = \alpha (a) \cdot \alpha (b) ~~~~ \text{ and } ~~~~ \alpha (a) \cdot (b \cdot c) = (a \cdot b) \cdot \alpha (c), \text{ for } a, b, c \in A.
    \end{align*}
    In this case, $A_\alpha^{ [~,~] }$ is a Hom-Lie algebra, where $[a, b] = a \cdot b - b \cdot a$, for $a, b \in A$.
\end{exam}

\begin{defn}
Let $\mathfrak{g}^{[~,~]}_\alpha$ be a Hom-Lie algebra. A {\bf representation} of this Hom-Lie algebra is a vector space $V$ equipped with linear maps $\rho : \mathfrak{g} \rightarrow \mathrm{End}(V)$ and $\beta : V \rightarrow V$ satisfying
\begin{align*}
& \qquad \qquad \beta \big( \rho (x) (v) \big) = \rho (\alpha (x)) (\beta (v)),\\
&\rho (\alpha (x)) \circ \rho (y) - \rho (\alpha (y)) \circ \rho(x) = \rho ([x,y]) \circ \beta,
\end{align*} 
for $x, y \in \mathfrak{g}$ and $v \in V$.
We denote a representation as above by $V^\rho_\beta$. 
\end{defn} 

It follows from the above definition that any Hom-Lie algebra $\mathfrak{g}^{[~,~]}_\alpha$ can be regarded as a representation of itself, where $\rho : \mathfrak{g} \rightarrow \mathrm{End}(\mathfrak{g})$ is given by $\rho (x) (y) = [x, y]$, for $x, y \in \mathfrak{g}$. This is called the 'adjoint representation'.

\medskip

\begin{defn}\label{hom-leib-defn}
A {\bf Hom-Leibniz algebra} is a triple $\mathfrak{h}^{ \{ ~,~ \} }_\alpha = (\mathfrak{h}, \{ ~,~ \}, \alpha)$ consisting of a vector space $\mathfrak{h}$ with a bilinear bracket $\{ ~,~ \} : \mathfrak{h} \times \mathfrak{h} \rightarrow \mathfrak{h}$ and a linear map $\alpha : \mathfrak{h} \rightarrow \mathfrak{h}$ that satisfying
\begin{align}\label{hom-leib-iden}
\alpha ( \{ x,y \}) = \{ \alpha (x) , \alpha (y) \} \quad \text{ and } \quad \{ \alpha (x) , \{ y,z \} \} = \{ \{ x,y \}, \alpha (z) \} + \{ \alpha (y), \{ x,z \} \}, \text{ for } x, y, z \in \mathfrak{h}.
\end{align}
\end{defn}

A Hom-Leibniz algebra whose bracket is skewsymmetric turns out to be a Hom-Lie algebra. Hence a Hom-Leibniz algebra can be realized as a non-skewsymmetric analogue of a Hom-Lie algebra.

\begin{exam}
Let $(\mathfrak{h}, \{ ~,~ \})$ be a Leibniz algebra and $\alpha : \mathfrak{h} \rightarrow \mathfrak{h}$ be a Leibniz algebra homomorphism. Then $\mathfrak{h}^{\alpha \circ \{ ~,~ \}}_\alpha = (\mathfrak{h}, \alpha \circ \{ ~,~ \}, \alpha)$ is a Hom-Leibniz algebra.
\end{exam}

\begin{exam}\label{hemi-exam}
Let $\mathfrak{g}^{ [ ~,~ ] }_\alpha$  be a Hom-Lie algebra and $V^\rho_\beta$ be a representation of it. Then $(\mathfrak{g} \oplus V)^{ \{ ~, ~\} }_{ \alpha \oplus \beta}$ is a Hom-Leibniz algebra, where
\begin{align*}
\{ (x, u), (y, v) \} = ([x, y], \rho (x) v), \text{ for } (x, u), (y, v) \in \mathfrak{g} \oplus V.
\end{align*}
This is called the hemi-semidirect product.
\end{exam}

Let $\mathfrak{h}^{ \{ ~,~ \} }_\alpha$ be a Hom-Leibniz algebra. A {\bf representation} of the Hom-Leibniz algebra $\mathfrak{h}^{ \{ ~,~ \}}_\alpha$ is a vector space $V$ equipped with two bilinear maps (called left and right $\mathfrak{h}$-actions, respectively)
\begin{align*}
\rho^L : \mathfrak{h} \times V \rightarrow V, ~(x, v) \mapsto \rho^L (x, v) \quad \text{ and } \quad \rho^R : V \times \mathfrak{h} \rightarrow V, ~(v,x) \mapsto \rho^R (v,x)
\end{align*}
and a linear map $\beta : V \rightarrow V$ satisfying the following set of identities: for $x, y \in \mathfrak{g}$ and $v \in V$,
\begin{align}
\beta (\rho^L(x,v)) = \rho^L (\alpha (x) & , \beta(v)), \quad \beta (\rho^R (v,x))= \rho^R (\beta (v) , \alpha(x)), \\
\rho^L (\alpha (x), \rho^L (y, v)) =~& \rho^L (\{ x,y \}, \beta (v)) + \rho^L (\alpha (y), \rho^L (x, v)),\\
\rho^L (\alpha (x), \rho^R (v, y)) =~& \rho^R (\rho^L (x, v), \alpha (y)) + \rho^R (\beta (v) , \{ x, y \} ),\\
\rho^R (\beta (v) , \{ x, y \}) =~& \rho^R (\rho^R (v, x), \alpha (y)) + \rho^L (\alpha (x), \rho^R (v, y)),
\end{align}
for $x, y \in \mathfrak{h}$ and $v \in V$. We denote a representation as above by $V^{\rho^L, \rho^R}_\beta$. It follows that any Hom-Leibniz algebra $\mathfrak{h}^{\{ ~,~ \}}_\alpha$ can be realized as a representation $\mathfrak{h}^{\{ ~,~ \}, \{ ~,~ \} }_\alpha$ of itself. We call this as the adjoint representation.

Let $\mathfrak{h}^{ \{ ~,~ \} }_\alpha$ be a Hom-Leibniz algebra and $V^{\rho^L, \rho^R}_\beta$ be a representation of it. Then the $n$-th cochain group $C^n (\mathfrak{h}_\alpha, V_\beta)$ of the Hom-Leibniz algebra $\mathfrak{h}^{\{ ~,~ \}}_\alpha$ with coefficients in $V^{\rho^L, \rho^R}_\beta$ is given by
\begin{align*}
C^n (\mathfrak{h}_\alpha, V_\beta ) = \begin{cases}
 \{ v \in V ~|~ \beta (v) = v \} & \text{ if } n=0,\\
 \mathrm{Hom} (\mathfrak{h}_\alpha^{\otimes n}, V_\beta) = \{ f \in \mathrm{Hom} (\mathfrak{h}^{\otimes n}, V) ~|~ \beta \circ f = f \circ \alpha^{\otimes n} \} & \text{ if } n \geq 1.
\end{cases}
\end{align*}
The coboundary map $\delta_\mathrm{HLeib} : C^n (\mathfrak{h}_\alpha, V_\beta ) \rightarrow C^{n+1} (\mathfrak{h}_\alpha, V_\beta )$ is given by
\begin{align}\label{leib-diff}
(\delta_\mathrm{HLeib} f) (x_1, \ldots, x_{n+1}) =~& \sum_{i=1}^n (-1)^{i+1} \rho^L \big(  \alpha^{n-1} (x_i), f (x_1, \ldots, \widehat{x_i}, \ldots, x_{n+1})  \big)  \\
&+ \sum_{1 \leq i < j \leq n+1} (-1)^i f \big(  \alpha (x_1), \ldots, \widehat{x_i}, \ldots, \alpha (x_{j-1}), \{ x_i, x_j \}, \alpha (x_{j+1}), \ldots, \alpha (x_{n+1})   \big) \nonumber  \\
&+ (-1)^{n-1} \rho^R \big( f (x_1, \ldots, x_n), \alpha^{n-1} (x_{n+1})   \big), \nonumber
\end{align}
for $f \in C^n (\mathfrak{h}_\alpha, V_\beta )$ and $x_i, \ldots, x_{n+1} \in \mathfrak{h}$. The cohomology group of the complex $\{ C^\bullet (\mathfrak{h}_\alpha, V_\beta ), \delta_\mathrm{HLeib} \}$ is called the {\bf Loday-Pirashvili cohomology} of the Hom-Leibniz algebra $\mathfrak{h}^{ \{ ~,~\}}_\alpha$ with coefficients in $V^{\rho^L, \rho^R}_\beta$ and denoted by $H^\bullet_\mathrm{HLeib} (\mathfrak{h}_\alpha^{ \{ ~, ~ \} }, V_\beta^{\rho^L, \rho^R })$.

\medskip

Let $\mathfrak{h}$ be a vector space and $\alpha : \mathfrak{h} \rightarrow \mathfrak{h}$ be a linear map. Then the graded vector space
\begin{align*}
\oplus_{n=1}^\infty C^n (\mathfrak{h}_\alpha, \mathfrak{h}_\alpha), \text{ where } C^{n \geq 1} (\mathfrak{h}_\alpha, \mathfrak{h}_\alpha) = \mathrm{Hom} (\mathfrak{h}_\alpha^{\otimes n}, \mathfrak{h}_\alpha)= \{ f \in \mathrm{Hom} (\mathfrak{h}^{\otimes n}, \mathfrak{h}) ~|~ \alpha \circ f = f \circ \alpha^{\otimes n} \}
\end{align*} 
carries a degree $-1$ bracket (called the {\bf Balavoine's bracket} \cite{bala}) given by
\begin{align}\label{bala-bracket}
[P, Q]_B := P \diamond Q - (-1)^{(m-1)(n-1)} Q \diamond P, \text{ where }
\end{align}
\begin{align}
(P \diamond Q)(x_1, \ldots, x_{m+n-1}) = \sum_{i=1}^m & \sum_{\sigma \in \mathbb{S}_{(i-1, n-1)}}  (-1)^{(i-1)(n-1)} (-1)^\sigma P \big( \alpha^{n-1} ( x_{\sigma (1)} ), \ldots, \alpha^{n-1} (x_{\sigma (i-1)}), \\ & Q ( x_{\sigma (i)}, \ldots, x_{\sigma (i+n-2)}, x_{i+n-1} ), \alpha^{n-1} ( x_{i+n}), \ldots, \alpha^{n-1} ( x_{m+n-1})   \big), \nonumber
\end{align}
for $P \in C^m (\mathfrak{h}_\alpha, \mathfrak{h}_\alpha)$, $Q \in C^n (\mathfrak{h}_\alpha, \mathfrak{h}_\alpha)$ and $x_1, \ldots, x_{m+n-1} \in \mathfrak{h}$. The Balavoine's bracket makes the shifted graded vector space $ C^{\bullet +1} (\mathfrak{h}_\alpha, \mathfrak{h}_\alpha) = \oplus_{n=0}^\infty C^{n+1} (\mathfrak{h}_\alpha, \mathfrak{h}_\alpha)$ into a graded Lie algebra. The following result gives the importance of this graded Lie algebra.

\begin{thm}
Let $\mathfrak{h}$ be a vector space and $\alpha: \mathfrak{h} \rightarrow \mathfrak{h}$ be a linear map. Then there is a one-to-one correspondence between Maurer-Cartan elements in the graded Lie algebra $( \oplus_{n=0}^\infty C^{n + 1} (\mathfrak{h}_\alpha, \mathfrak{h}_\alpha), [~,~]_B)$ and bilinear brackets $\{ ~,~ \} : \mathfrak{h} \times \mathfrak{h} \rightarrow \mathfrak{h}$ which make $\mathfrak{h}^{ \{ ~,~\} }_\alpha$ into a Hom-Leibniz algebra.
\end{thm}

Let $\mathfrak{h}^{ \{ ~, ~ \} }_\alpha$ be a Hom-Leibniz algebra. Then the bracket $\{ ~, ~ \} : \mathfrak{h} \times \mathfrak{h} \rightarrow \mathfrak{h}$ defines an element $\pi \in C^2 (\mathfrak{h}_\alpha, \mathfrak{h}_\alpha)$ given by
\begin{align*}
    \pi (x, y) = \{ x, y \}, \text{ for } x, y \in \mathfrak{h}.
\end{align*}
It is easy to see that the Hom-Leibniz identity (the second identity of (\ref{hom-leib-iden})) of the bracket $\{ ~, ~ \}$ is equivalent to $[\pi, \pi]_B = 0$. In other words, $\pi$ is a Maurer-Cartan element in the graded Lie algebra  $( \oplus_{n=0}^\infty C^{n + 1} (\mathfrak{h}_\alpha, \mathfrak{h}_\alpha), [~,~]_B)$. With these notations, the coboundary map (\ref{leib-diff}) of the Hom-Leibniz algebra $\mathfrak{h}^{ \{ ~, ~ \} }_\alpha$ with coefficients in the adjoint representation is given by
\begin{align*}
    \delta_\mathrm{HLeib} (f) =  (-1)^{n-1} [\pi, f]_B, \text{ for } f \in C^{n \geq 1} (\mathfrak{h}_\alpha, \mathfrak{h}_\alpha).
\end{align*}


\medskip

\section{Embedding tensors on Hom-Lie algebras and their cohomology}\label{sec3}
In this section, we introduce embedding tensors on Hom-Lie algebras. First, we show that an embedding tensor naturally gives rise to a Hom-Leibniz algebra structure. Next, we construct the derived graded Lie algebra that characterizes embedding tensors as its Maurer-Cartan elements. Finally, we define the cohomology of an embedding tensor and study some consequences of the cohomology.


\begin{defn}
Let $\mathfrak{g}^{[~,~]}_\alpha$ be a Hom-Lie algebra and $V^\rho_\beta$ be a representation of it. A linear map $T: V \rightarrow \mathfrak{g}$ is said to be an {\bf embedding tensor} (on $\mathfrak{g}^{[~,~]}_\alpha$ with respect to the representation $V^\rho_\beta$) if
\begin{align}\label{embed-iden}
\alpha \circ T = T \circ \beta ~~~~ \text{ and } ~~~~ [T(u), T(v)] = T \big( \rho (Tu) v \big), \text{ for } u, v \in V.
\end{align}
\end{defn}

\begin{exam}
Let $\mathfrak{g}^{[~,~]}_\alpha$ be a Hom-Lie algebra. Then the identity map $\mathrm{id}: \mathfrak{g} \rightarrow \mathfrak{g}$ is an embedding tensor on the Hom-Lie algebra $\mathfrak{g}^{[~,~]}_\alpha$ with respect to the adjoint representation.
\end{exam}

\begin{exam}
Let $\mathfrak{g}^{[~,~]}_\alpha$ be a Hom-Lie algebra. A linear map $d: \mathfrak{g} \rightarrow \mathfrak{g}$ is said to be a derivation for the Hom-Lie algebra $\mathfrak{g}^{[~,~]}_\alpha$ if $\alpha \circ d = d \circ \alpha$ and 
$
d [x,y] = [dx, y] + [x, dy], \text{ for all } x, y \in \mathfrak{g}.
$
If $d$ is a square-zero derivation, then $d$ is an embedding tensor on $\mathfrak{g}^{ [~,~]}_\alpha$ with respect to the adjoint representation.
\end{exam}

\begin{exam}
Let $\mathfrak{g}^{[~,~]}_\alpha$ be a Hom-Lie algebra. Then it can be easily checked that $(\underbrace{\mathfrak{g} \oplus \cdots \oplus \mathfrak{g}}_{n\mathrm{~ times}})^\rho_{\alpha \oplus \cdots \oplus \alpha}$ is a representation of the Hom-Lie algebra $\mathfrak{g}^{[~,~]}_\alpha$, where 
\begin{align*}
\rho : \mathfrak{g} \rightarrow \mathrm{End} (\mathfrak{g} \oplus \cdots \oplus \mathfrak{g}), ~~ \rho (x) \big(  (x_1, \ldots, x_n)  \big) = ([x, x_1], \ldots, [x, x_n]), 
\end{align*}
for $ x \in \mathfrak{g}$ and $(x_1, \ldots, x_n) \in \mathfrak{g} \oplus \cdots \oplus \mathfrak{g}.$ Then the map $T : \mathfrak{g} \oplus \cdots \oplus \mathfrak{g} \rightarrow \mathfrak{g}$ is an embedding tensor, where
\begin{align*}
T \big( (x_1, \ldots, x_n ) \big) = x_1 + \cdots + x_n, \text{ for } (x_1, \ldots, x_n) \in \mathfrak{g} \oplus \cdots \oplus \mathfrak{g}.
\end{align*}
\end{exam}

\begin{exam}
With the notations of the previous example, the $i$-th projection map 
\begin{align*}
T_i : \mathfrak{g} \oplus \cdots \oplus \mathfrak{g} \rightarrow \mathfrak{g},~~ T_i \big( (x_1, \ldots, x_n) \big) = x_i
\end{align*}
 is an embedding tensor on $\mathfrak{g}^{ [~,~]}_\alpha$ with respect to the representation $(\underbrace{\mathfrak{g} \oplus \cdots \oplus \mathfrak{g}}_{n\mathrm{~ times}})^\rho_{\alpha \oplus \cdots \oplus \alpha}$.
\end{exam}

\begin{exam}
Let $\mathfrak{g}^{[~,~]}_\alpha$ be a Hom-Lie algebra and $V^\rho_\beta$ be a representation of it. Let $f : V \rightarrow \mathfrak{g}$ be a $\mathfrak{g}$-equivariant map, i.e., $f$ is a linear map satisfying
\begin{align*}
\alpha \circ f = f \circ \beta ~~~~ \text{ and } ~~~~ f (\rho (x) v) = [x, f(v)], \text{ for } x \in \mathfrak{g}, v \in V.
\end{align*}
Then it is easy to see that $f$ is an embedding tensor on $\mathfrak{g}^{[~,~]}_\alpha$ with respect to the representation $V^\rho_\beta$.
\end{exam}

\begin{exam}
(Crossed modules of Hom-Lie algebras \cite{sheng-hl}) A crossed module of Hom-Lie algebras is a quadruple $\big(  (\mathfrak{g}_1)^{[~,~]_1}_{\alpha_1}, (\mathfrak{g}_2)^{[~,~]_2}_{\alpha_2}, d, \rho \big)$ in which $(\mathfrak{g}_1)^{[~,~]_1}_{\alpha_1}, ~(\mathfrak{g}_2)^{[~,~]_2}_{\alpha_2}$ are Hom-Lie algebras, $d: \mathfrak{g}_2 \rightarrow \mathfrak{g}_1$ is a morphism of Hom-Lie algebras and $\rho : \mathfrak{g}_1 \rightarrow \mathrm{End}(\mathfrak{g}_2)$ is a linear map that makes $(\mathfrak{g}_2)^\rho_{\alpha_2}$ into a representation of the Hom-Lie algebra $(\mathfrak{g}_1)_{\alpha_1}^{[~,~]_1}$, satisfying additionally
\begin{align}\label{cr-mod}
    d (\rho (x) m) = [x, dm]_1 ~~~~ \text{ and } ~~~~~ \rho (dm) n = [m, n]_2, \text{ for } x \in \mathfrak{g}_1, m, n \in \mathfrak{g}_2.
\end{align}
It follows from the first condition of (\ref{cr-mod}) that the map $d$ is an embedding tensor on the Hom-Lie algebra $(\mathfrak{g}_1)^{[~,~]_1}_{\alpha_1}$ with respect to the representation $(\mathfrak{g}_2)^\rho_{\alpha_2}$.
\end{exam}

In \cite{sheng-hl} the authors introduced Hom-Lie $2$-algebras and showed that `strict' Hom-Lie $2$-algebras correspond to crossed modules of Hom-Lie algebras. Hence, by following our previous example, one could construct embedding tensors from strict Hom-Lie $2$-algebras.

\begin{defn}\label{em-ten-mor}
Let $T: V \rightarrow \mathfrak{g}$ be an embedding tensor on a Hom-Lie algebra $\mathfrak{g}$ with respect to the representation $V_\beta^\rho$. Let $T'$ be another embedding tensor on the Hom-Lie algebra $(\mathfrak{g}')_{\alpha'}^{[~,~]'}$ with respect to the representation $(V')_{\beta'}^{\rho'}$. A {\bf homomorphism} of embedding tensors from $T$ to $T'$ is given by a pair $(\varphi, \psi)$ in which $\varphi : \mathfrak{g} \rightarrow \mathfrak{g}'$ is a Hom-Lie algebra morphism, $\psi : V \rightarrow V'$ is a linear map satisfying
\begin{align*}
   \beta' \circ \psi = \psi \circ \beta , \quad \psi (\rho (x) v) = \rho' (\varphi (x)) (\psi (v)) ~~~ \text{ and } ~~~ \varphi \circ T = T' \circ \psi, \text{ for } x \in \mathfrak{g}, v \in V.
\end{align*}
\end{defn}

\medskip

In the following result, we characterize embedding tensors by their graphs. This characterization can be used to construct the induced Hom-Leibniz algebra associated with an embedding tensor.

\begin{prop}
Let $\mathfrak{g}^{[~,~]}_\alpha$ be a Hom-Lie algebra and $V^\rho_\beta$ be a representation of it. A linear map $T: V \rightarrow \mathfrak{g}$ is an embedding tensor (on the Hom-Lie algebra $\mathfrak{g}^{[~,~]}_\alpha$ with respect to the representation $V^\rho_\beta$) if and only if the graph
\begin{align*}
Gr (T) = \{ (T(u), u)~|~ u \in V   \} \subset \mathfrak{g} \oplus V
\end{align*}
is a Hom-Leibniz subalgebra of the hemi-semidirect product $(\mathfrak{g} \oplus V)^{   \{ ~,~ \} }_{\alpha \oplus \beta}$ (see Example \ref{hemi-exam}).
\end{prop}

\begin{proof}
Note that $Gr(T)$ is a Hom-Leibniz subalgebra of  $(\mathfrak{g} \oplus V)^{   \{ ~,~ \} }_{\alpha \oplus \beta}$ if and only if
\begin{align*}
(\alpha \oplus \beta) (Gr (T)) \subset Gr (T) ~~~~ \text{ and } ~~~~ \{ Gr(T), Gr(T) \} \subset Gr (T).
\end{align*}
For any $u \in V$, we consider the element $(T(u), u) \in Gr(T)$. Then 
\begin{align*}
(\alpha \oplus \beta) (T(u), u) = (\alpha T (u), \beta (u)) \in Gr(T) \text{ if and only if } \alpha \circ T = T \circ \beta.
\end{align*}
Similarly, for any $u, v \in V$, we have
\begin{align*}
\{ (T(u), u), (T(v), v) \} = ([T(u), T(v)], \rho (Tu) v).
\end{align*}
This is in $Gr(T)$ if and only if the second identity of (\ref{embed-iden}) holds. Hence $Gr(T)$ is a Hom-Leibniz subalgebra of $(\mathfrak{g} \oplus V)^{ \{ ~,~ \} }_{\alpha \oplus \beta}$ if and only if $T$ is an embedding tensor.
\end{proof}

Since $Gr (T)$ can be identified with the space $V$, we have the following consequence from our previous proposition.

\begin{prop}\label{induced-hl}
    Let $T: V \rightarrow \mathfrak{g}$ be an embedding tensor on a Hom-Lie algebra with respect to the representation $V^\rho_\beta$. Then $V_\beta^{ \{ ~, ~ \} }$ is a Hom-Leibniz algebra (called the induced Hom-Leibniz algebra), where the bracket $\{ ~, ~ \}: V \times V \rightarrow V$ is given by
    \begin{align*}
        \{ u, v \} := \rho (Tu) v, \text{ for } u, v \in V.
    \end{align*}
\end{prop}

\medskip

In the following, we give another characterization of embedding tensors. Namely, given a Hom-Lie algebra and a representation of it, we construct a graded Lie algebra whose Maurer-Cartan elements correspond to embedding tensors. Using the help of such characterization, we also define the cohomology of an embedding tensor. Finally, we realize the cohomology of an embedding tensor as the Loday-Pirashvili cohomology of the induced Hom-Leibniz algebra with coefficients in a suitable representation.

Let $\mathfrak{g}^{ [ ~, ~ ] }_\alpha$ be a Hom-Lie algebra and $V^\rho_\beta$ be a representation of it. We consider the direct sum vector space $ \mathfrak{h} = \mathfrak{g} \oplus V$ with the linear map $\alpha \oplus \beta:  \mathfrak{g} \oplus V \rightarrow  \mathfrak{g} \oplus V$. Let
\begin{align*}
    \big(   \oplus_{n=0}^\infty C^{n+1} ( (\mathfrak{g} \oplus V)_{\alpha \oplus \beta},  (\mathfrak{g} \oplus V)_{\alpha \oplus \beta}) , [~,~]_B \big)
\end{align*}
be the graded Lie algebra with Balavoine's bracket. Since the hemi-semidirect product $(\mathfrak{g} \oplus V)^{ \{ ~, ~ \} }_{\alpha \oplus \beta}$ is a Hom-Leibniz algebra, it gives rise to an element $\pi \in C^2 (  (\mathfrak{g} \oplus V)_{\alpha \oplus \beta},  (\mathfrak{g} \oplus V)_{\alpha \oplus \beta})  )$ that satisfies $[\pi, \pi]_B = 0$. Explicitly, the element $\pi$ is given by
\begin{align}\label{defn-pi}
    \pi \big(  (x,u), (y,v)  \big) = ([x,y], \rho (x) v), \text{ for } (x,u), (y,v) \in \mathfrak{g} \oplus V.
\end{align}
Further, it is easy to check that the graded subspace
\begin{align*}
    \oplus_{n=0}^\infty C^{n+1} (V_\beta, \mathfrak{g}_\alpha) ~ \subset ~ \oplus_{n=0}^\infty C^{n+1} ( (\mathfrak{g} \oplus V)_{\alpha \oplus \beta},  (\mathfrak{g} \oplus V)_{\alpha \oplus \beta})
\end{align*}
is an abelian Lie subalgebra. Thus, the shifted (shifted by degree $1$) subspace $ \oplus_{n=1}^\infty C^{n} (V_\beta, \mathfrak{g}_\alpha)$ inherits a graded Lie algebra structure with the Kosmann-Schwarzbach's derived bracket \cite{yks}. In our case, the bracket is given by the following theorem.

\begin{thm}
Let $\mathfrak{g}^{ [ ~, ~ ]}_\alpha$ be a Hom-Lie algebra and $V^\rho_\beta$ be a representation of it. Then the graded vector space $C^\bullet (V_\beta, \mathfrak{g}_\alpha) = \oplus_{n=1}^\infty C^{n} (V_\beta, \mathfrak{g}_\alpha)$ with the bracket
\begin{align}\label{deri-br}
\llbracket P, Q \rrbracket = (-1)^{m-1} [[ \pi, P]_B, Q]_B, \text{ for } P \in C^m(V_\beta, \mathfrak{g}_\alpha), Q \in C^n(V_\beta, \mathfrak{g}_\alpha)
\end{align}
is a graded Lie algebra. (This is called the derived graded Lie algebra).
\end{thm}


The bracket (\ref{deri-br}) is explicitly given by the formula
\begin{align}\label{exp-deri-br}
&\llbracket P, Q \rrbracket (v_1, \ldots, v_{m+n}) \\
&=  \sum_{i=1}^m \sum_{\sigma \in \mathbb{S}_{(i-1,n)}} (-1)^{(i-1)n+1} (-1)^\sigma P \big(  \beta^{n} (v_{\sigma (1)}), \ldots, \beta^{n} (v_{\sigma (i-1)}), \nonumber \\ 
& \qquad \qquad  \qquad \qquad \qquad \qquad \qquad \qquad   \rho \big(   Q (v_{\sigma (i)}, \ldots, v_{\sigma (i+n-1)}) (\beta^{n-1} (v_{i+n}))  \big) , \ldots, \beta^{n} (v_{\sigma (m+n)})   \big) \nonumber \\
&- (-1)^{mn}  \sum_{i=1}^n \sum_{\sigma \in \mathbb{S}_{(i-1,m)}} (-1)^{(i-1)m+1} (-1)^\sigma Q \big(  \beta^{m} (v_{\sigma (1)}), \ldots, \beta^{m} (v_{\sigma (i-1)}), \nonumber \\ 
& \qquad \qquad  \qquad \qquad \qquad \qquad    \qquad \qquad \rho \big(   P (v_{\sigma (i)}, \ldots, v_{\sigma (i+m-1)}) (\beta^{m-1}(v_{i+m}))  \big) , \ldots, \beta^{m} (v_{\sigma (m+n)})   \big) \nonumber \\
& + \sum_{\sigma \in \mathbb{S}_{(m,n)}} (-1)^{mn+1} (-1)^\sigma \big[ P \big( 
    \beta^{n-1} (v_{\sigma (1)}), \ldots, \beta^{n-1} (v_{\sigma (m)}) \big), Q \big(      \beta^{m-1} (v_{\sigma (m+1)}), \ldots, \beta^{m-1} (v_{\sigma (m+n)})   \big) \big], \nonumber
\end{align}
for $v_1, \ldots, v_{m+n} \in V$. The following result gives the importance of the above-graded Lie algebra.
\begin{thm}\label{embd-mc}
Let $\mathfrak{g}^{ [ ~, ~ ] }_\alpha$ be a Hom-Lie algebra and $V^\rho_\beta$ be a representation of it.  A linear map $T: V \rightarrow \mathfrak{g}$ that satisfies $\alpha \circ T = T \circ \beta$ is an embedding tensor if and only if  $T \in C^1(V_\beta, \mathfrak{g}_\alpha)$ is a Maurer-Cartan element in the derived graded Lie algebra $(C^\bullet(V_\beta, \mathfrak{g}_\alpha), \llbracket ~, ~ \rrbracket).$
\end{thm}

\begin{proof}
For any $T \in C^1(V_\beta, \mathfrak{g}_\alpha)$, we have from (\ref{exp-deri-br}) that
\begin{align*}
\llbracket T, T \rrbracket (u, v) = 2 \big( [T(u), T(v)] - T (\rho (Tu) v)   \big), \text{ for } u, v \in V.
\end{align*}
This shows that $T$ is an embedding tensor if and only if $\llbracket T, T \rrbracket = 0$. This completes the proof.
\end{proof}

\medskip

Let $\mathfrak{g}^{[ ~, ~ ]}_\alpha$ be a Hom-Lie algebra, $V^\rho_\beta$ be a representation and $T: V \rightarrow \mathfrak{g}$ be an embedding tensor. Then it follows from the previous theorem that $\llbracket T, T \rrbracket  = 0$. Hence $T$ induces a degree $1$ differential
\begin{align*}
d_T : C^n (V_\beta, \mathfrak{g}_\alpha) \rightarrow C^{n + 1} (V_\beta, \mathfrak{g}_\alpha), ~ d_T = \llbracket T, - \rrbracket, \text{ for } n \geq 1.
\end{align*}
We define $C^0 (V_\beta, \mathfrak{g}_\alpha) = \{ x \in \mathfrak{g} ~|~ \alpha (x) = x \}$ and a map (also denoted by the same notation as above) $d_T : C^0 (V_\beta, \mathfrak{g}_\alpha) \rightarrow C^1 (V_\beta, \mathfrak{g}_\alpha)$ by
\begin{align*}
    d_T (x) (v) = [x, Tv] - T ( \rho (x) v), \text{ for } x \in C^0 (V_\beta, \mathfrak{g}_\alpha), v \in V.
\end{align*}
Then one can easily check that the composition $C^0 (V_\beta, \mathfrak{g}_\alpha) \xrightarrow{d_T} C^1 (V_\beta, \mathfrak{g}_\alpha) \xrightarrow{d_T} C^2 (V_\beta, \mathfrak{g}_\alpha)$ is the zero map.
In other words, one obtains a cochain complex $\{ C^\bullet (V_\beta, \mathfrak{g}_\alpha), d_T \}$. The corresponding cohomology is denoted by $H^\bullet_T (V_\beta, \mathfrak{g}_\alpha)$ and it is called the {\bf cohomology} of the embedding tensor $T$.

\medskip

Let $T: V \rightarrow \mathfrak{g}$ be an embedding tensor on the Hom-Lie algebra $\mathfrak{g}^{ [ ~, ~ ] }_\alpha$ with respect to the representation $V^\rho_\beta$. Then we have seen in Proposition \ref{induced-hl} that $V^{ \{ ~, \} }_\beta$ is a Hom-Leibniz algebra, where $\{ u, v \} =  \rho (Tu)v$, for $u, v \in V$. We define two bilinear maps $\overline{\rho}^L : V \times \mathfrak{g} \rightarrow \mathfrak{g}$ and $\overline{\rho}^R : \mathfrak{g} \times V \rightarrow \mathfrak{g}$ by
\begin{align*}
\overline{\rho}^L (v, x) = [Tv, x] ~~~~ \text{ ~ and ~ } ~~~~ \overline{\rho}^R (x, v) = [x, Tv] - T (\rho (x) v), \text{ for } v \in V, x \in \mathfrak{g}.
\end{align*}
With the above notations, we have the following.

\begin{prop}
The maps $\overline{\rho}^L, \overline{\rho}^R$ define a representation of the Hom-Leibniz algebra $V^{ \{ ~, ~\} }_\beta$ on the vector space $\mathfrak{g}$ with the linear map $\alpha : \mathfrak{g} \rightarrow \mathfrak{g}$. In other words, $\mathfrak{g}^{ \overline{\rho}^L, \overline{\rho}^R}_\alpha$ is a representation of the induced Hom-Leibniz algebra $V^{ \{ ~, ~ \}}_\beta$.
\end{prop}

The proof of the above proposition follows from a straightforward calculation. Hence we omit the details. Now, it follows from the proposition that one may consider the Loday-Pirashvili cohomology of the induced Hom-Leibniz algebra $V^{ \{ ~, ~ \}}_\beta$ with coefficients in the representation $\mathfrak{g}^{\overline{\rho}^L, \overline{\rho}^R}_\alpha$. More precisely, we consider the cochain complex $\{ C^\bullet (V_\beta, \mathfrak{g}_\alpha), \delta_\mathrm{HLeib} \}$, where
\begin{align*}
C^0(V_\beta, \mathfrak{g}_\alpha) = \{ x \in \mathfrak{g} ~|~ \alpha (x) = x \}, \quad C^{n \geq 1} (V_\beta, \mathfrak{g}_\alpha) = \mathrm{Hom}(V_\beta^{\otimes n}, \mathfrak{g}_\alpha).
\end{align*}
The coboundary map $\delta_\mathrm{HLeib} : C^n(V_\beta, \mathfrak{g}_\alpha) \rightarrow C^{n+1}(V_\beta, \mathfrak{g}_\alpha)$ is given by
\begin{align*}
(\delta_\mathrm{HLeib} f) (v_1, \ldots, v_{n+1}) =~& \sum_{i=1}^n (-1)^{i+1} ~ [ T ( \beta^{n-1} (v_i) ), f (v_1, \ldots, \widehat{v_i}, \ldots, v_{n+1})  ]  \\
+& \sum_{1 \leq i < j \leq n+1} (-1)^i f \big(  \beta (v_1), \ldots, \widehat{v_i}, \ldots, \beta (v_{j-1}), \rho ( T v_i) v_j , \beta (v_{j+1}), \ldots, \beta (v_{n+1})   \big) \nonumber  \\
+& (-1)^{n-1} ~\big( [ f (v_1, \ldots, v_n), T (\beta^{n-1} (v_{n+1}))] - T \big(  \rho (  f (v_1, \ldots, v_n)  ) (\beta^{n-1} (v_{n+1}))  \big)   \big), \nonumber
\end{align*}
for $f \in C^n ( V_\beta, \mathfrak{g}_\alpha)$ and $v_1, \ldots, v_{n+1} \in V$. We denote by $H^\bullet_\mathrm{HLeib} (V_\beta^{ \{ ~, ~ \} }, \mathfrak{g}_\alpha^{ \overline{\rho}^L, \overline{\rho}^R } )$ the corresponding cohomology groups.

\begin{prop}
Let $T: V \rightarrow \mathfrak{g}$ be an embedding tensor on the Hom-Lie algebra $\mathfrak{g}^{ [~,~]}_\alpha$ with respect to the representation $V^\rho_\beta$. Then the coboundary maps $d_T$ and $\delta_\mathrm{HLeib}$ are related by
\begin{align*}
d_T (f) = (-1)^{n-1} ~ \delta_\mathrm{HLeib} (f), \text{ for } f \in C^n (V_\beta, \mathfrak{g}_\alpha).
\end{align*}
In other words, we have $H^\bullet_T (V_\beta, \mathfrak{g}_\alpha) 
  \cong  H^\bullet_\mathrm{HLeib} ({V}_\beta^{ \{ ~, ~ \}},  \mathfrak{g}_\alpha^{ \overline{\rho}^L, \overline{\rho}^R } ).$
\end{prop}

\begin{proof}
For any $f \in C^n (V_\beta, \mathfrak{g}_\alpha)$ and $v_1, \ldots, v_{n+1} \in V$, it follows from (\ref{exp-deri-br}) that
\begin{align*}
   & (d_T (f)) (v_1, \ldots, v_{n+1}) \\
   & = \llbracket T, f \rrbracket (v_1, \ldots, v_{n+1}) \\
   & = - T \big(   \rho (   f(v_1, \ldots, v_n)) (\beta^{n-1}(v_{n+1})) \big) \\
   & \quad - (-1)^n \sum_{i=1}^n \sum_{\sigma \in \mathbb{S}_{(i-1,1)}} (-1)^i (-1)^\sigma  f \big( \beta (v_{\sigma (1)}), \ldots, \beta (v_{\sigma (i-1)}) , \rho (  T v_{\sigma (i)}) v_{i+1}    , \ldots, \beta (v_{n+1}) \big) \\
   & \quad  + \sum_{\sigma \in \mathbb{S}_{(1,n)}} (-1)^{n+1} (-1)^\sigma [T (\beta^{n-1} (v_{\sigma (1)})), f (  
 v_{\sigma (2)}, \ldots, v_{\sigma (n+1)})]  \\
 &= - T \big(   \rho (   f(v_1, \ldots, v_n)) (\beta^{n-1}(v_{n+1})) \big) \\ 
 & \quad - (-1)^n  \sum_{1 \leq i < j \leq n+1} (-1)^i f \big(  \beta (v_1), \ldots, \widehat{v_i}, \ldots, \beta (v_{j-1}), \rho ( T v_i) v_j , \beta (v_{j+1}), \ldots, \beta (v_{n+1})   \big) \\
 & \quad + \sum_{i=1}^n (-1)^{n+1} (-1)^{i-1} [ T (\beta^{n-1} (v_i)), f (v_1, \ldots, \widehat{v_i}, \ldots, v_{n+1})] -  [ T (\beta^{n-1} (v_{n+1})), f (v_1, \ldots, v_{n})] \\
 &= (-1)^{n-1} (\delta_\mathrm{HLeib} (f)) (v_1, \ldots, v_{n+1}).
\end{align*}
This completes the proof.
\end{proof}


It follows from the above proposition that the cohomology of the embedding tensor $T$ can be realized as the Loday-Pirashvili cohomology of the induced Hom-Leibniz algebra $V^{\{ ~, ~\} }_\beta$ with coefficients in the representation $\mathfrak{g}^{ \overline{\rho}^L, \overline{\rho}^R}_\alpha$.

\medskip

In the following, we show that there is a homomorphism from the cohomology of an embedding tensor to the cohomology of the induced Hom-Leibniz algebra with coefficients in the adjoint representation. Let $\mathfrak{g}^{[ ~, ~ ]}_\alpha$ be a Hom-Lie algebra and $V^\rho_\beta$ be a representation of it. We consider the following two graded Lie algebras, namely, 

(1) the derived graded Lie algebra $(C^\bullet (V_\beta, \mathfrak{g}_\alpha), \llbracket ~, ~ \rrbracket)$, whose Maurer-Cartan elements are embedding tensors,

(2) the graded Lie algebra $(C^{\bullet +1} (V_\beta, V_\beta), [~, ~]_B)$ with the Balavoine's bracket, whose Maurer-Cartan elements are precisely Hom-Leibniz algebra brackets on $V_\beta$.

For any $n \geq 1$, we define a map $\Phi_n : C^n(V_\beta, \mathfrak{g}_\alpha) \rightarrow C^{n+1} (V_\beta,V_\beta)$ by
\begin{align*}
\Phi_n (f) (v_1, \ldots, v_{n+1}) = - \rho \big(  f(v_1, \ldots, v_n)  \big) ( \beta^{n-1}(v_{n+1})), \text{ for } f \in C^n(V_\beta, \mathfrak{g}_\alpha) \text{ and } u_1, \ldots, v_{n+1} \in V.
\end{align*}
With the above notations, we have the following.

\begin{prop}
The collection of maps $\{ \Phi_n \}_{n \geq 1}$ defines a homomorphism of graded Lie algebras from $( C^\bullet (V_\beta, \mathfrak{g}_\alpha), \llbracket ~, ~ \rrbracket )$ to $( C^{\bullet + 1} (V_\beta, V_\beta), [~,~]_B )$.
\end{prop}

\begin{proof}
Let $P \in C^m (V_\beta, \mathfrak{g}_\alpha)$ and $Q \in C^n (V_\beta, \mathfrak{g}_\alpha)$ be two arbitrary homogeneous elements of $C^\bullet (V_\beta, \mathfrak{g}_\alpha)$. Then for any $v_1, \ldots, v_{m+n+1} \in V$, we have
\begin{align*}
    &\Phi_m (P) \diamond \Phi_n (Q) \\
    &= \sum_{i=1}^{m+1} \sum_{\sigma \in \mathbb{S}_{(i-1,n)}} (-1)^{(i-1)n} (-1)^{ \sigma} ~ \Phi_m (P) 
    \bigg(  \beta^n v_{\sigma (1)}, \ldots, \beta^n v_{\sigma (i-1)}, \\ 
    & \qquad \qquad \qquad  \qquad \qquad \qquad  \qquad \qquad \qquad \qquad \Phi_n (Q) \big( v_{\sigma (i)}, \ldots, v_{\sigma (i+n-1)}, v_{i+n} \big), \ldots, \beta^n v_{m+n+1}  \bigg) \\
    &= \sum_{i=1}^m \sum_{\sigma \in \mathbb{S}_{(i-1,n)}} (-1)^{(i-1)n} (-1)^\sigma ~\rho \bigg(   P \big(    \beta^n v_{\sigma (1)}, \ldots, \beta^n v_{\sigma (i-1)}, \\
   & \qquad \qquad \qquad \qquad \qquad  \rho \big(       Q ( v_{\sigma (i)}, \ldots, v_{\sigma (i+n-1)}) (\beta^{n-1} v_{i+n})    \big), \ldots, \beta^n v_{m+n}  \big)   \bigg)  (\beta^{m+n-1} v_{m+n+1}) \\
   & + \sum_{\sigma \in \mathbb{S}_{(m,n)}} (-1)^{mn + \sigma} \rho \big(    P (  \beta^n v_{\sigma(1)}, \ldots, \beta^n v_{\sigma(m)}) \big)   \rho \big( Q (  \beta^{m-1} v_{\sigma (m+1)}, \ldots, \beta^{m-1} v_{\sigma(m+n)}  )   \big) (\beta^{m+n-2} v_{m+n+1}).
\end{align*}
Similarly, 
\begin{align*}
    &\Phi_n (Q) \diamond \Phi_m (P) \\
    &= \sum_{i=1}^{n+1} \sum_{\sigma \in \mathbb{S}_{(i-1,m)}} (-1)^{(i-1)m} (-1)^{ \sigma} ~ \Phi_n (Q) 
    \bigg(  \beta^m v_{\sigma (1)}, \ldots, \beta^m v_{\sigma (i-1)}, \\ 
    & \qquad \qquad \qquad  \qquad \qquad \qquad  \qquad \qquad \qquad \qquad \Phi_m (P) \big( v_{\sigma (i)}, \ldots, v_{\sigma (i+m-1)}, v_{i+m} \big), \ldots, \beta^m v_{m+n+1}  \bigg) \\
    &= \sum_{i=1}^n \sum_{\sigma \in \mathbb{S}_{(i-1,m)}} (-1)^{(i-1)m} (-1)^\sigma ~\rho \bigg(   Q \big(    \beta^m v_{\sigma (1)}, \ldots, \beta^m v_{\sigma (i-1)}, \\
   & \qquad \qquad \qquad \qquad \qquad  \rho \big(       P ( v_{\sigma (i)}, \ldots, v_{\sigma (i+m-1)}) (\beta^{m-1} v_{i+m})    \big), \ldots, \beta^m v_{m+n}  \big)   \bigg)  (\beta^{m+n-1} v_{m+n+1}) \\
   & ~~ + \sum_{\sigma \in \mathbb{S}_{(n,m)}} (-1)^{mn + \sigma} \rho \big(    Q (  \beta^m v_{\sigma(1)}, \ldots, \beta^m v_{\sigma(n)}) \big)   \rho \big( P (  \beta^{n-1} v_{\sigma (n+1)}, \ldots, \beta^{n-1} v_{\sigma(m+n)}  )   \big) (\beta^{m+n-2} v_{m+n+1})\\
   &= \sum_{i=1}^n \sum_{\sigma \in \mathbb{S}_{(i-1,m)}} (-1)^{(i-1)m} (-1)^\sigma ~\rho \bigg(   Q \big(    \beta^m v_{\sigma (1)}, \ldots, \beta^m v_{\sigma (i-1)}, \\
   & \qquad \qquad \qquad \qquad \qquad  \rho \big(       P ( v_{\sigma (i)}, \ldots, v_{\sigma (i+m-1)}) (\beta^{m-1} v_{i+m})    \big), \ldots, \beta^m v_{m+n}  \big)   \bigg)  (\beta^{m+n-1} v_{m+n+1}) \\
   & ~~ + \sum_{\tau \in \mathbb{S}_{(m,n)}} (-1)^{ \tau} \rho \big(    Q (  \beta^m v_{\tau(m+1)}, \ldots, \beta^m v_{\tau(m+n)}) \big)   \rho \big( P (  \beta^{n-1} v_{\tau (1)}, \ldots, \beta^{n-1} v_{\tau(m)}  )   \big) (\beta^{m+n-2} v_{m+n+1}).
\end{align*}
Here, in the last term, we observed that a shuffle $\sigma \in \mathbb{S}_{(n,m)}$ corresponds to shuffle $\tau \in \mathbb{S}_{(m,n)}$ given by 
\begin{align*}
    \tau (i) = \sigma (n+i) \text{ for } 1 \leq i \leq m \quad \text{ and } \quad \tau (i) = \sigma (i-m) \text{ for } m+1 \leq i \leq m+n.
\end{align*}
This correspondence in fact gives rise to a bijection between the sets $\mathbb{S}_{(n,m)}$ and $\mathbb{S}_{(m,n)}$. Hence, we have
\begin{align*}
  &  [\Phi_m (P), \Phi_n (Q)]_B (v_1, \ldots, v_{m+n+1}) \\
 &= \big( \Phi_m (P) \diamond \Phi_n (Q) - (-1)^{mn} \Phi_n (Q) \diamond \Phi_m (P) \big) (v_1, \ldots, v_{m+n+1}) \\
 &= \sum_{i=1}^m \sum_{\sigma \in \mathbb{S}_{(i-1,n)}} (-1)^{(i-1)n} (-1)^\sigma ~\rho \bigg(   P \big(    \beta^n v_{\sigma (1)}, \ldots, \beta^n v_{\sigma (i-1)}, \\
   & \qquad \qquad \qquad \qquad \qquad  \rho \big(       Q ( v_{\sigma (i)}, \ldots, v_{\sigma (i+n-1)}) (\beta^{n-1} v_{i+n})    \big), \ldots, \beta^n v_{m+n}  \big)   \bigg)  (\beta^{m+n-1} v_{m+n+1}) \\
   & - (-1)^{mn} \sum_{i=1}^n \sum_{\sigma \in \mathbb{S}_{(i-1,m)}} (-1)^{(i-1)m} (-1)^\sigma ~\rho \bigg(   Q \big(    \beta^m v_{\sigma (1)}, \ldots, \beta^m v_{\sigma (i-1)}, \\
   & \qquad \qquad \qquad \qquad \qquad  \rho \big(       P ( v_{\sigma (i)}, \ldots, v_{\sigma (i+m-1)}) (\beta^{m-1} v_{i+m})    \big), \ldots, \beta^m v_{m+n}  \big)   \bigg)  (\beta^{m+n-1} v_{m+n+1}) \\
   &+ \sum_{\sigma \in \mathbb{S}_{(m,n)}} (-1)^{mn +\sigma} \rho \big(  [ P ( \beta^{n-1} v_{\sigma (1)}, \ldots, \beta^{n-1} v_{\sigma (m)}), Q (\beta^{m-1} v_{\sigma (m+1)}, \ldots, \beta^{m-1} v_{\sigma (m+n)})       ]     \big) (\beta^{m+n-1} v_{m+n+1}) \\
 &= \big(  \Phi_{m+n} \llbracket P, Q \rrbracket  \big) (v_1, \ldots, v_{m+n+1}).
\end{align*}
This shows that the collection $\{ \Phi_n \}_{n \geq 1}$ defines a homomorphism of graded Lie algebras.
\end{proof}

As a consequence of the above proposition, we get the following result.

\begin{thm}
Let $\mathfrak{g}^{[~,~]}_\alpha$ be a Hom-Lie algebra and $V^\rho_\beta$ be a representation of it. Let $T: V \rightarrow \mathfrak{g}$ be an embedding tensor (on $\mathfrak{g}^{[~,~]}_\alpha$ with respect to the representation $V^\rho_\beta$). Then the collection $\{ \Phi_n \}_{n \geq 1}$ induces a homomorphism
\begin{align*}
\Phi_\bullet : H^\bullet_T (V_\beta, \mathfrak{g}_\alpha) \rightarrow H^{\bullet +1}_\mathrm{HLeib} (V_\beta^{ \{ ~, ~\} }, V_\beta^{ \{ ~, ~\}, \{ ~, ~ \} })
\end{align*}
from the cohomology of the embedding tensor $T$ to the Loday-Pirashvili cohomology of the induced Leibniz algebra $V_\beta^{ \{ ~, ~\} }$ with coefficients in the adjoint representation.
\end{thm}

\begin{proof}
Since the collection $\{ \Phi_n \}_{n \geq 1}$ defines a homomorphism of graded Lie algebras from $( C^\bullet (V_\beta, \mathfrak{g}_\alpha), \llbracket ~, ~ \rrbracket)$ to $(C^{\bullet +1} (V_\bullet, V_\beta), [~,~]_B)$, it follows that the map $\Phi_1$ maps Maurer-Cartan elements of the first graded Lie algebra to Maurer-Cartan elements of the second one. Note that the cohomology $H^\bullet_T (V_\beta, \mathfrak{g}_\alpha)$ is induced by the Maurer-Cartan element $T$ of the first graded Lie algebra and (up to some sign) the cohomology $ H^{\bullet +1}_\mathrm{HLeib} (V_\beta^{ \{ ~, ~\} }, V_\beta^{ \{ ~, ~\}, \{ ~, ~ \} })$ is induced by the Maurer-Cartan element $\Phi_1 (T) = - \{ ~, ~ \}$ of the second graded Lie algebra. This concludes the proof.
\end{proof}

\medskip

Note that the differential $d_T$ (of the cochain complex of the embedding tensor $T$) is a derivation for the graded Lie bracket $\llbracket ~, ~ \rrbracket$. To see this, we observe that
\begin{align*}
d_T (\llbracket P, Q \rrbracket) = \llbracket T, \llbracket P, Q \rrbracket \rrbracket 
=~& \llbracket \llbracket T, P \rrbracket, Q \rrbracket + (-1)^m \llbracket P, \llbracket T, Q \rrbracket \rrbracket \\
=~& \llbracket d_T (P), Q \rrbracket + (-1)^m \rrbracket P, d_T (Q) \rrbracket, 
\end{align*}
for $P \in C^m(V_\beta, \mathfrak{g}_\alpha)$ and $Q \in C^n (V_\beta, \mathfrak{g}_\alpha)$. In other words, we obtain a differential graded Lie algebra.

\begin{prop}\label{mc-defor}
Let $T : V \rightarrow \mathfrak{g}$ be an embedding tensor (on the Hom-Lie algebra $\mathfrak{g}^{ [ ~, ~ ] }_\alpha$ with respect to the representation $V^\rho_\beta$). Then for any linear map $T_1 : V \rightarrow \mathfrak{g}$ with $\alpha \circ T_1 = T_1 \circ \beta$, the sum $T+ T_1$ is an embedding tensor if and only if $T_1 \in C^1 (V_\beta, \mathfrak{g}_\alpha)$ is a Maurer-Cartan element in the differential graded Lie algebra $( C^\bullet (V_\beta, \mathfrak{g}_\alpha), \llbracket ~, ~ \rrbracket, d_T)$.
\end{prop}

\begin{proof}
Note that $T + T_1$ is an embedding tensor if and only if $\llbracket T + T_1, T + T_1 \rrbracket = 0$. This holds if and only if
\begin{align*}
    \llbracket  T,T \rrbracket + \llbracket  T,T_1 \rrbracket + \llbracket  T_1,T \rrbracket + \llbracket  T_1,T_1 \rrbracket = 0.
\end{align*}
This condition is further equivalent to $d_T (T_1) + \frac{1}{2} \llbracket  T_1,T_1 \rrbracket = 0$, i.e. $T_1$ is a Maurer-Cartan element in the differential graded Lie algebra $( C^\bullet (V_\beta, \mathfrak{g}_\alpha), \llbracket ~, ~ \rrbracket, d_T)$. Hence the proof.
\end{proof}

It follows from the above proposition that the differential graded Lie algebra $( C^\bullet (V_\beta, \mathfrak{g}_\alpha), \llbracket ~, ~ \rrbracket, d_T)$ controls the deformations of the embedding tensor $T$. For this reason, the differential graded Lie algebra $( C^\bullet (V_\beta, \mathfrak{g}_\alpha), \llbracket ~, ~ \rrbracket, d_T)$ is called the controlling structure for the deformations of $T$.

\begin{defn}
Let $T : V \rightarrow \mathfrak{g}$ be an embedding tensor (on $\mathfrak{g}^{ [ ~, ~ ] }_\alpha$ with respect to the representation $V^\rho_\beta$). An {\bf infinitesimal deformation} of $T$ is an embedding tensor of the form $T + \epsilon T_1$, where $\epsilon$ is a parameter with $\epsilon^2 = 0$.
\end{defn}

Let $T+ \epsilon T_1$ be an infinitesimal deformation of $T$. It follows from Proposition \ref{mc-defor} that $T_1 \in C^1 (V_\beta, \mathfrak{g}_\alpha)$ is a $1$-cocycle in the cohomology complex of $T$. Hence the cohomology class of $T_1$ defines an element in $H^1_T (V_\beta, \mathfrak{g}_\alpha)$.

Let $T+ \epsilon T_1$ and $T+ \epsilon T_1'$ be two infinitesimal deformations of $T$. They are said to be {\bf equivalent} if there exists an $\alpha$-invariant element $a \in \mathfrak{g}$ (i.e. $\alpha (a) = a$) such that the pair of maps $(\mathrm{id}_\mathfrak{g} + \epsilon [a, -], \mathrm{id}_V + \epsilon \rho (a))$ defines a homomorphism of embedding tensors from $T+ \epsilon T_1$ to $T+ \epsilon T_1'$. In other words, the following conditions must hold:

(i) the linear map $(\mathrm{id}_\mathfrak{g} + \epsilon [a, -]) : \mathfrak{g} \rightarrow \mathfrak{g}$ is a Hom-Lie algebra morphism,

(ii) the linear map $(\mathrm{id}_V + \epsilon \rho (a)) : V \rightarrow V$ satisfies
\begin{align}
    (\mathrm{id}_V + \epsilon \rho (a)) \circ \beta =~& \beta \circ (\mathrm{id}_V + \epsilon \rho (a)), \\
    (\mathrm{id}_V + \epsilon \rho (a)) \big(   \rho(x)v \big) =~& \rho \big( (\mathrm{id}_\mathfrak{g} + \epsilon [a, -])(x)  \big) (\mathrm{id}_V + \epsilon \rho (a)) (v), \\
    (\mathrm{id}_\mathfrak{g} + \epsilon [a, -]) \circ (T + \epsilon T_1) (v) =~& (T + \epsilon T_1') \circ (\mathrm{id}_V + \epsilon \rho (a))(v), \label{t-commute}
\end{align}
for all $x \in \mathfrak{g}$ and $v \in V$. It is easy to see that the condition (\ref{t-commute}) gives rise to
\begin{align*}
    T_1' (v) - T_1 (v) = [a, T(v)] - T (\rho (a) v) = (\delta_T (a))(v).
\end{align*}
This shows that $T_1$ and $T_2$ are cohomologous. Hence their cohomology classes are the same in $ H^1_T  (V_\beta, \mathfrak{g}_\alpha).$

Conversely, any $1$-cocycle $T_1$ gives rise to the infinitesimal deformation $T+ \epsilon T_1$. Moreover, cohomologous $1$-cocycles correspond to equivalent infinitesimal deformations. As a summary of the above discussions, we obtain the following result.


\begin{thm}
Let $T: V \rightarrow \mathfrak{g}$ be an embedding tensor. Then there is a bijection between the set of all equivalence classes of infinitesimal deformations of $T$ and the first cohomology group $H^1_T (V_\beta, \mathfrak{g}_\alpha).$
\end{thm}


\medskip

\section{Hom-Lie-Leibniz triples}\label{sec5}
In this section, we first consider Hom-Lie-Leibniz triples as the full underlying structure of embedding tensors. We show that the category of Hom-Leibniz algebras is closely related to the category of Hom-Lie-Leibniz triples. Finally, we construct an $L_\infty$-algebra (using Voronov's higher derived brackets) whose Maurer-Cartan elements correspond to Hom-Lie-Leibniz triples.

\begin{defn}
    A {\bf Hom-Lie-Leibniz triple} is a triple $(\mathfrak{g}_\alpha^{[~,~]}, V_\beta^\rho, T)$ in which $\mathfrak{g}_\alpha^{[~,~]}$ is a Hom-Lie algebra, $V_\beta^\rho$ is a representation and $T: V \rightarrow\mathfrak{g}$ is an embedding tensor.
\end{defn}

It follows from the above definition that a Hom-Lie-Leibniz triple captures the full information (including the underlying structures) of an embedding triple. Let $(\mathfrak{g}_\alpha^{[~,~]}, V_\beta^\rho, T)$ and $((\mathfrak{g}')_{\alpha'}^{[~,~]'}, (V')_{\beta'}^{\rho'}, T')$ be two Hom-Lie-Leibniz triples. A {\bf homomorphism} of Hom-Lie-Leibniz triples from $(\mathfrak{g}_\alpha^{[~,~]}, V_\beta^\rho, T)$ to $((\mathfrak{g}')_{\alpha'}^{[~,~]'}, (V')_{\beta'}^{\rho'}, T')$ is given by a homomorphism of embedding tensors from $T$ to $T'$ (cf. Definition \ref{em-ten-mor}). The set of all Hom-Lie-Leibniz triples and homomorphisms between them forms a category, denoted by {\bf HLLTriple}.

\begin{prop}
    (i) Let $(\mathfrak{g}_\alpha^{[~,~]}, V_\beta^\rho, T)$ be a Hom-Lie-Leibniz triple. Then $V_\beta^{ \{ ~, ~ \} }$ is a Hom-Leibniz algebra, where $\{ u, v \} := \rho (Tu) v,$ for $u, v \in V$.

    (ii) Let $(\mathfrak{g}_\alpha^{[~,~]}, V_\beta^\rho, T)$, $((\mathfrak{g}')_{\alpha'}^{[~,~]'}, (V')_{\beta'}^{\rho'}, T')$ be two Hom-Lie-Leibniz triples and $(\varphi, \psi)$ be a homomorphism between them. Then $\psi : V \rightarrow V'$ is a homomorphism of induced Hom-Leibniz algebras from $V_\beta^{ \{ ~, ~ \} }$ to $(V')_{\beta'}^{ \{ ~, ~ \}' }$.
\end{prop}

\begin{proof}
    Part (i) follows from Proposition \ref{induced-hl}. For part (ii), we observe that
    \begin{align*}
         \psi ( \{ u, v\}) = \psi \big(   \rho (Tu) v \big) = \rho' \big(  \varphi T(u)  \big) (\psi (v)) = \rho' \big(  T' \psi (u)   \big) (\psi (v)) = \{ \psi (u) , \psi (v) \}'
    \end{align*}
    for $u, v \in V$. This proves the result.
\end{proof}

It follows from the above proposition that there is a functor $\mathcal{F} : {\bf HLLTriple} \rightarrow {\bf HLeib}$ from the category of Hom-Lie-Leibniz triples to the category of Hom-Leibniz algebras. In the following, we will construct a functor in the other direction. 

Let $\mathfrak{h}_\alpha^{ \{ ~, ~ \} }$ be a Hom-Leibniz algebra. Take $\overline{\mathfrak h} = \mathfrak{h}/ I$, where $I$ is the ideal of the Hom-Leibniz algebra generated by the elements $\{ x, y \}$, for $x, y \in \mathfrak{h}$. Note that the linear map $\alpha : \mathfrak{h} \rightarrow \mathfrak{h}$ induces a map (denoted by the same notation) $\alpha : \overline{\mathfrak h} \rightarrow \overline{\mathfrak h}$. We define a bilinear skewsymmetric bracket $[~,~] : \overline{\mathfrak h} \times \overline{\mathfrak h} \rightarrow \overline{\mathfrak h}$ by
\begin{align*}
    [\langle x \rangle, \langle y \rangle] := \langle \{x, y \} \rangle, \text{ for } \langle x \rangle, \langle y \rangle \in \overline{\mathfrak h}.
\end{align*}
Here $\langle x \rangle \in \overline{\mathfrak h}$ denotes the class of the element $x \in \mathfrak{h}$. It turns out that  $(\overline{\mathfrak h})_\alpha^{[~,~]}$ is a Hom-Lie algebra. We define another map $\rho : \overline{\mathfrak h} \rightarrow \mathrm{End} (\mathfrak{h})$ by $\rho (\langle x \rangle) (y) := \{ x , y \},$ for $\langle x \rangle \in \overline{\mathfrak h}$ and $y \in \mathfrak{h}$. It is easy to verify that the map $\rho$ makes $\mathfrak{h}^\rho_\alpha$ into a representation of the Hom-Lie algebra $(\overline{\mathfrak h})^{[~,~]}_\alpha$. With these notations, the quotient map $q: \mathfrak{h} \rightarrow \overline{\mathfrak h} = \mathfrak{h}/I$ is an embedding tensor on the Hom-Lie algebra $(\overline{\mathfrak h})^{[~,~]}_\alpha$ with respect to the representation $\mathfrak{h}_\alpha^\rho$. In other words, $( (\overline{\mathfrak h})^{[~,~]}_\alpha , \mathfrak{h}^\rho_\alpha, q  )$ is a Hom-Lie-Leibniz triple. Moreover, the induced Hom-Leibniz algebra structure coincides with the given one.

Let $\mathfrak{h}_\alpha^{\{ ~, ~ \}}$ and $(\mathfrak{h}')_{\alpha'}^{ \{ ~, ~ \}' }$ be two Hom-Leibniz algebras and $\psi : \mathfrak{h} \rightarrow \mathfrak{h}'$ be a homomorphism of Hom-Leibniz algebras. Then a straightforward calculation shows that the pair of maps
\begin{align*}
    (\overline{\psi}, \psi) : ( (\overline{\mathfrak h})^{[~,~]}_\alpha , \mathfrak{h}^\rho_\alpha, q  ) \rightsquigarrow ( (\overline{\mathfrak{h}'})^{[~,~]'}_{\alpha'} , (\mathfrak{h}')^{\rho'}_{\alpha'}, q'  ) 
\end{align*}
defines a homomorphism of Hom-Lie-Leibniz triples, where the map  $\overline{\psi} : \overline{\mathfrak h} \rightarrow \overline{\mathfrak{h}'}$ is given by $\overline{\psi} (\langle x \rangle) = \langle \psi (x) \rangle$, for $\langle x \rangle \in \overline{h}$. Hence, we obtain a functor $\mathcal{G}: {\bf HLeib} \rightarrow {\bf HLLTriple}$ from the category of Hom-Leibniz algebras to the category of Hom-Lie-Leibniz triples.

\medskip

Next, we give the Maurer-Cartan characterization of Hom-Lie-Leibniz triples. We first need the definition and Voronov's construction of an $L_\infty$-algebra.
\begin{defn}\label{l-inf-defn}
An {\bf $L_\infty$-algebra} (strongly homotopy Lie algebra) is a pair $(L = \oplus_{i \in \mathbb{Z}} L_i, \{ l_k \}_{k \geq 1})$ consisting of a graded vector space $L = \oplus_{i \in \mathbb{Z} }L_i$ equipped with a collection $\{ l_k: L^{\otimes k} \rightarrow L \}_{k \geq 1}$ of degree $1$ graded linear maps satisfying the following set of identities:

- (graded symmetry) $l_k (x_{\sigma (1)}, \ldots, x_{\sigma (k)}) = \varepsilon(\sigma) l_k (x_1, \ldots, x_k),$ for all $k\geq 1$, $\sigma \in \mathbb{S}_k$,

- (higher Jacobi identities) for any $n \geq 1$ and homogeneous elements $x_1, \ldots, x_n \in L$,
\begin{align*}
\sum_{i+j=n+1} \sum_{\sigma \in \mathbb{S}_{(i, n-i)}} \varepsilon (\sigma)~ l_j \big(   l_i ( x_{\sigma (1)}, \ldots, x_{\sigma (i)}), x_{\sigma (i+1)}, \ldots, x_{\sigma (n)} \big) = 0.
\end{align*}
Here $\varepsilon(\sigma)$ is the standard Koszul sign that appears in the graded context.
\end{defn}

Let $(L, \{ l_k \}_{k \geq 1})$ be a $L_\infty$-algebra. An element $\theta \in L_0$ is said to be a {\bf Maurer-Cartan element} of the $L_\infty$-algebra if
\begin{align}\label{mc-iden}
\sum_{k=1}^\infty \frac{1}{k!} l_k (\underbrace{\theta, \ldots, \theta}_{k \text{ times}}) = 0.
\end{align}
Note that the identity (\ref{mc-iden}) makes sense when the infinite sum converges. If $\theta$ is a Maurer-Cartan element of the $L_\infty$-algebra $(L, \{ l_k \}_{k \geq 1})$, then the pair $(L, \{ l_k^\theta \}_{k \geq 1})$ is also an $L_\infty$-algebra, where
\begin{align*}
l_k^\theta (x_1, \ldots, x_k) = \sum_{n=0}^\infty \frac{1}{n!} ~ l_{n+k} (\underbrace{\theta, \ldots, \theta}_{n \text{ times}}, x_1, \ldots, x_k), \text{ for } k \geq 1 \text{ and } x_1, \ldots, x_k \in L.
\end{align*}
The $L_\infty$-algebra $(L, \{ l_k^\theta \}_{k \geq 1})$ is said to obtained from $(L, \{ l_k \}_{k \geq 1})$ twisted by the Maurer-Cartan element $\theta$. For any $\theta' \in L_0$, the element $\theta + \theta'$ is a Maurer-Cartan element of the $L_\infty$-algebra $(L, \{ l_k \}_{k \geq 1})$ if and only if $\theta'$ is a Maurer-Cartan element of the twisted $L_\infty$-algebra $(L, \{ l_k^\theta \}_{k \geq 1})$.

There is a useful construction of an $L_\infty$-algebra given by Voronov \cite{voro}. His construction relies on the use of higher derived brackets. We first recall the construction. 

\begin{defn}\label{defn-v-d}
A {\bf $V$-data} is a quadruple $(\mathfrak{l}, \mathfrak{a}, P, \Delta)$ in which $\mathfrak{l}$ is a graded Lie algebra (with the bracket $[~,~]$), $\mathfrak{a} \subset \mathfrak{l}$ is an abelian Lie subalgebra, $P : \mathfrak{l} \rightarrow \mathfrak{l}$ is a projection with $\mathrm{im} (P) = \mathfrak{a}$ and $\mathrm{ker}(P) \subset \mathfrak{l}$ a graded Lie algebra, and $\Delta \in \mathrm{ker}(P)_1$ that satisfies $[\Delta, \Delta] = 0$.
\end{defn}

\begin{thm}\label{thm-voro}  Let $(\mathfrak{l}, \mathfrak{a}, P, \Delta)$ be a $V$-data. Then for any graded Lie subalgebra $\mathfrak{q} \subset \mathfrak{l}$ with $[\Delta, \mathfrak{q}] \subset \mathfrak{q}$, the graded vector space $s^{-1} \mathfrak{q} \oplus \mathfrak{a}$ inherits an $L_\infty$-algebra structure with the operations $\{ l_k \}_{k \geq 1}$  given by the following: for homogeneous elements $x, y \in \mathfrak{q}$ and $a, a_1, \ldots, a_k \in \mathfrak{a}$,
\begin{align*}
l_1 \big(  (s^{-1}x, a)  \big) =~&  (-s^{-1}[\Delta, x], P (x + [\Delta, a])), \\
l_2 \big( (s^{-1}x, 0), (s^{-1}y, 0)   \big) =~& ( (-1)^{|x|} s^{-1} [x,y], 0), \\
l_k \big(   (0, a_1), \ldots, (0, a_k) \big) =~& (0, P [ \cdots [[ \Delta, a_1], a_2], \ldots, a_k ]), \text{ for } k \geq 1, \\
l_k \big( (s^{-1} x, 0), (0, a_1), \ldots, (0, a_{k-1}) \big) =~&   (0, P [ \cdots [[ x, a_1], a_2], \ldots, a_{k-1}]), \text{ for } k \geq 2,\\
\text{ and up to permutations of the} & \text{ above inputs, all other maps vanish.}
\end{align*} 
\end{thm}

In the following, we will use the above theorem to construct an $L_\infty$-algebra whose Maurer-Cartan elements are precisely Hom-Lie-Leibniz triples. Let $\mathfrak{g}$ and $V$ be two vector spaces equipped with linear maps $\alpha: \mathfrak{g} \rightarrow \mathfrak{g}$ and $\beta: V \rightarrow V$ (not necessarily with any additional structures). We consider the linear map $\alpha \oplus \beta: \mathfrak{g} \oplus V \rightarrow \mathfrak{g} \oplus V$ on the direct sum vector space. Let
\begin{align*}
\mathfrak{l} =  \big( C^{\bullet +1} ( (\mathfrak{g} \oplus V)_{\alpha \oplus \beta}, (\mathfrak{g} \oplus V )_{\alpha \oplus \beta}) = \oplus_{n=0}^\infty C^{n+1} ( (\mathfrak{g} \oplus V)_{\alpha \oplus \beta}, (\mathfrak{g} \oplus V)_{\alpha \oplus \beta} ), [~, ~]_B  \big)
\end{align*}
be the graded Lie algebra with Balavoine's bracket. Then it is easy to check that the graded subspace $\mathfrak{a} = \oplus_{n=0}^\infty C^{n+1} (V_\beta, \mathfrak{g}_\alpha)$ is an abelian Lie subalgebra of $\mathfrak{l}$. Let $P: \mathfrak{l} \rightarrow \mathfrak{l}$ be the projection onto the subspace $\mathfrak{a}$. Then the quadruple $(\mathfrak{l}, \mathfrak{a}, P, \Delta = 0)$ is a $V$-data.

For each $n \geq 0$, we define
\begin{align*}
C^{n,1} (\mathfrak{g}_\alpha, V_\beta)  = \mathrm{Hom}(\wedge^n \mathfrak{g}_\alpha \otimes V_\beta , V_\beta) = \{ f \in \mathrm{Hom} (\wedge^n \mathfrak{g} \otimes V, V ) ~|~ f \circ (\wedge^n \alpha \otimes \beta) = \beta \circ f \}.
\end{align*}
With the above notations, we have the following.
\begin{prop}
The graded subspace 
\begin{align}\label{defn-q}
\mathfrak{q} = \oplus_{n=0}^\infty \big( \mathrm{ Hom} (\wedge^{n+1} \mathfrak{g}_\alpha, \mathfrak{g}_\alpha) \oplus C^{{n,1}} (\mathfrak{g}_\alpha, V_\beta)  \big) \subset \mathfrak{l}
\end{align}
is a graded Lie subalgebra of $\mathfrak{l}$.
\end{prop}

A map $f \in C^{n+1} ( (\mathfrak{g} \oplus V)_{\alpha \oplus \beta} , (\mathfrak{g} \oplus V)_{\alpha \oplus \beta})$ is said to have the bidegree $n | 0$ if
\begin{align*}
    f \in \mathrm{Hom} (\wedge^{n+1} \mathfrak{g}_\alpha, \mathfrak{g}_\alpha) \oplus \mathrm{Hom} (\wedge^n \mathfrak{g}_\alpha \otimes V_\beta, V_\beta).
\end{align*}
If $f$ has bidegree $n|0$ and $g$ has bidegree $m|0$, then it is a straightforward calculation to verify that $[f,g]_B$ has bidegree $(m+n)|0$ (see \cite{sheng-embd}). This observation yields proof of the above proposition. Thus, as a consequence of Theorem \ref{thm-voro}, we obtain the following.

\begin{thm}
Let $\mathfrak{g}, V$ be two vector spaces equipped with linear maps $\alpha : \mathfrak{g} \rightarrow \mathfrak{g}$ and $\beta : V \rightarrow V$. Then there is an $L_\infty$-algebra structure on the graded vector space $s^{-1}{ \mathfrak{q}} \oplus \mathfrak{a}$ with the structure maps $\{ l_k \}_{k \geq 1}$ given by
\begin{align*}
l_2 ( (s^{-1} q , 0), (s^{-1} q', 0)) =~& ( (-1)^{|q|} s^{-1} [q, q']_B, 0),\\
l_k ( (s^{-1} q , 0), (0, p_1) , \ldots, (0, p_{k-1}) ) =~& (0, P [ \cdots [[ q, p_1]_B, p_2]_B, \ldots, p_{k-1}]_B), \text{ for } k \geq 2.
\end{align*}
Here $q, q' \in \mathfrak{q}$ (which we consider as elements $s^{-1}q$, $s^{-1} q' \in s^{-1} \mathfrak{q}$) and $p_1, \ldots, p_{k-1} \in \mathfrak{a}$ are homogeneous elements. Up to permutations of the above inputs, all other maps vanish.
\end{thm}

Note that the graded vector space $s^{-1} \mathfrak{q} \oplus \mathfrak{a}$ is explicitly given by
\begin{align*}
s^{-1} \mathfrak{q} \oplus \mathfrak{a} = \oplus_{i=0}^\infty (s^{-1} \mathfrak{q} \oplus \mathfrak{a})_i = \oplus_{i=0}^\infty \big(  \mathrm{ Hom} (\wedge^{i+2} \mathfrak{g}_\alpha, \mathfrak{g}_\alpha)  \oplus C^{i+1,1} (\mathfrak{g}_\alpha, V_\beta) \oplus C^{i+1} (V_\beta, \mathfrak{g}_\alpha)  \big).
\end{align*}
The following result shows that the above $L_\infty$-algebra characterizes Hom-Lie-Leibniz triples as its Maurer-Cartan elements.

\begin{thm}
Let $\mathfrak{g}, V$ be two vector spaces equipped with linear maps $\alpha : \mathfrak{g} \rightarrow \mathfrak{g}$ and $\beta : V \rightarrow V$. Suppose there are linear maps
\begin{align*}
\mu \in \mathrm{Hom}(\wedge^2 \mathfrak{g}_\alpha, \mathfrak{g}_\alpha), ~~~~ \rho \in \mathrm{Hom} (\mathfrak{g}_\alpha \otimes V_\beta , V_\beta) ~~~~ \text{ and } ~~~~ T \in \mathrm{Hom} (V_\beta, \mathfrak{g}_\alpha).
\end{align*}
Then $(\mathfrak{g}_\alpha^\mu, V^\rho_\beta, T)$ is a Hom-Lie-Leibniz triple if and only if the element $\theta = (s^{-1} (\mu \boxplus \rho), T) \in (s^{-1} \mathfrak{q} \oplus \mathfrak{a})_0$ is a Maurer-Cartan element in the $L_\infty$-algebra $(s^{-1} \mathfrak{q} \oplus \mathfrak{a}, \{ l_k \}_{k \geq 1})$. Here $\mu \boxplus \rho \in \mathrm{ Hom} (\wedge^{2} \mathfrak{g}_\alpha, \mathfrak{g}_\alpha)  \oplus C^{1,1} (\mathfrak{g}_\alpha, V_\beta)$ is the element given by
\begin{align*}
    (\mu \boxplus \rho) \big(  (x,u), (y, v)  \big) = (\mu (x, y), \rho (x, v)), \text{ for } (x, u), (y, v) \in \mathfrak{g} \oplus V.
\end{align*}
\end{thm}

\begin{proof}
Note that 
\begin{align*}
    &\sum_{k=1}^\infty \frac{1}{k!} ~ l_k \big( (s^{-1} (\mu \boxplus \rho), T), \ldots, (s^{-1} (\mu \boxplus \rho), T)   \big)   \\
    &= \frac{1}{2!} ~l_2 \big(  (s^{-1} (\mu \boxplus \rho), T), (s^{-1} (\mu \boxplus \rho), T)  \big) + \frac{1}{3!} ~ l_3 \big(  (s^{-1} (\mu \boxplus \rho), T), (s^{-1} (\mu \boxplus \rho), T), (s^{-1} (\mu \boxplus \rho), T) \big) \\
    &= \big(  -s^{-1} \frac{1}{2} [  \mu \boxplus \rho, \mu \boxplus \rho   ]_B , ~ \frac{1}{2} [[ \mu \boxplus \rho, T]_B, T]_B  \big).
\end{align*}
This shows that $\theta = (s^{-1} (\mu \boxplus \rho), T)$ is a Maurer-Cartan element of the $L_\infty$-algebra if and only if
\begin{align*}
     [  \mu \boxplus \rho, \mu \boxplus \rho   ]_B  = 0 ~~~~ \text{ and } ~~~~  [[ \mu \boxplus \rho, T]_B, T]_B  = 0.
\end{align*}
The condition $[\mu \boxplus \rho, \mu \boxplus \rho ]_B  = 0$ is equivalent to the fact that $\mathfrak{g}_\alpha^\mu$ is a Hom-Lie algebra and $V_\beta^\rho$ is a representation of it. The second condition $ [[ \mu \boxplus \rho, T]_B, T]_B  = 0$ is equivalent that $T$ is an embedding tensor on the Hom-Lie algebra $\mathfrak{g}_\alpha^\mu$ with respect to the representation $V_\beta^\rho$ (cf. Theorem \ref{embd-mc}). Combining both the above facts, we get the result.
\end{proof}

In the next section, we will use the above Maurer-Cartan characterization of a Hom-Lie-Leibniz triple to define the cohomology. We will also see an application of the cohomology.

Let $(\mathfrak{g}^{[~,~]}_\alpha , V^\rho_\beta, T)$ be a Hom-Lie-Leibniz triple. Then it follows from the previous theorem that $\theta = (s^{-1} (\mu \boxplus \rho), T)$ is a Maurer-Cartan element of the $L_\infty$-algebra $(s^{-1}\mathfrak{q} \oplus \mathfrak{a}, \{ l_k \}_{k \geq 1})$, where $\mu = [~,~]$. Hence one can construct the $L_\infty$-algebra $\big( s^{-1}\mathfrak{q} \oplus \mathfrak{a}, \{ l_k^{(s^{-1}(\mu \boxplus \rho), T)} \}_{k \geq 1} \big)$ twisted by the Maurer-Cartan element $\theta = (s^{-1} (\mu \boxplus \rho), T)$. This $L_\infty$-algebra is called the controlling algebra of the given Hom-Lie-Leibniz triple $(\mathfrak{g}^{[~,~]}_\alpha , V^\rho_\beta, T)$. The justification is given by the following result.

\begin{thm}
Let $(\mathfrak{g}^{[~,~]}_\alpha , V^\rho_\beta, T)$ be a Hom-Lie-Leibniz triple. Then for any linear maps
\begin{align*}
\mu' \in \mathrm{Hom}(\wedge^2 \mathfrak{g}_\alpha, \mathfrak{g}_\alpha), ~~~ \rho' \in \mathrm{Hom}(\mathfrak{g}_\alpha \otimes V_\beta, V_\beta) ~~~ \text{ and } ~~~ T' \in \mathrm{Hom}(V_\beta, \mathfrak{g}_\alpha),
\end{align*}
the triple $(\mathfrak{g}_\alpha^{[~,~] + [~,~]'}, V_\beta^{\rho + \rho'}, T+ T')$ is a Hom-Lie-Leibniz triple if and only if $\theta' = (s^{-1} (\mu' \boxplus \rho'), T')$ is a Maurer-Cartan element of the controlling $L_\infty$-algebra $\big( s^{-1}\mathfrak{q} \oplus \mathfrak{a}, \{ l_k^{(s^{-1}(\mu \boxplus \rho), T)} \}_{k \geq 1} \big)$.
\end{thm}

\section{Cohomology and deformations of Hom-Lie-Leibniz triples}\label{sec6}
In this section, we first define the cohomology of a Hom-Lie-Leibniz triple. Given a Hom-Lie-Leibniz triple, we show that there is a long exact sequence connecting various cohomology groups. Finally, as an application of our cohomology, we study infinitesimal deformations of a Hom-Lie-Leibniz triple.

Let $(\mathfrak{g}_\alpha^{[~,~]}, V_\beta^\rho)$ be a pair consisting of a Hom-Lie algebra $\mathfrak{g}_\alpha^{[~,~]}$ and a representation $V_\beta^\rho$. We call such a pair a Hom-LieRep pair. Thus, a Hom-Lie-Leibniz triple $(\mathfrak{g}_\alpha^{[~,~]}, V_\beta^\rho, T)$ is a Hom-LieRep pair $(\mathfrak{g}_\alpha^{[~,~]}, V_\beta^\rho)$ equipped with an embedding tensor $T$. Here we first construct a cohomology theory associated with a Hom-LieRep pair. We will see that the cohomology of the Hom-Lie-Leibniz triple $(\mathfrak{g}_\alpha^{[~,~]}, V_\beta^\rho, T)$ can be seen as a byproduct of this cohomology with the cohomology of the embedding tensor $T$ (as defined in Section \ref{sec3}).

Let $(\mathfrak{g}_\alpha^{[~,~]}, V_\beta^\rho)$ be a Hom-LieRep pair. For each $n \geq 0$, we define the space $C^n_\mathrm{HLR} (\mathfrak{g}^{[~,~]}_\alpha, V^\rho_\beta)$ of $n$-cochains by
\begin{align*}
    C^0_\mathrm{HLR} ( \mathfrak{g}^{[~,~]}_\alpha, V^\rho_\beta ) = 0 \quad \text{ and } \quad  C^{n \geq 1}_\mathrm{HLR} ( \mathfrak{g}^{[~,~]}_\alpha, V^\rho_\beta ) = \mathrm{Hom} (\wedge^n \mathfrak{g}_\alpha, \mathfrak{g}_\alpha) \oplus \mathrm{Hom} (\wedge^{n-1} \mathfrak{g}_\alpha \otimes V_\beta , V_\beta).
\end{align*}
In other words, $C^{n \geq 1}_\mathrm{HLR} ( \mathfrak{g}^{[~,~]}_\alpha, V^\rho_\beta ) = \mathfrak{q}_{n-1}$ (see (\ref{defn-q})). We define a map $\delta_\mathrm{HLR} : C^n_\mathrm{HLR} ( \mathfrak{g}^{[~,~]}_\alpha, V^\rho_\beta ) \rightarrow C^{n+1}_\mathrm{HLR} ( \mathfrak{g}^{[~,~]}_\alpha, V^\rho_\beta)$ by
\begin{align}
    \delta_\mathrm{HLR} ((f_\mathfrak{g} , f_V)) = (-1)^{n-1} [\pi, f_\mathfrak{g} + f_V ]_B, \text{ for } (f_\mathfrak{g} , f_V) \in C^n_\mathrm{HLR} ( \mathfrak{g}^{[~,~]}_\alpha, V^\rho_\beta ),
\end{align}
where $\pi$ is given in (\ref{defn-pi}). The explicit description of the map $\delta_\mathrm{HLR}$ is given by 
\begin{align*}
\delta_\mathrm{HLR}  ((f_\mathfrak{g} , f_V)) = (\delta (f_\mathfrak{g}) , \delta^{f_\mathfrak{g}} (f_V)),
\end{align*}
 where the components $\delta (f_\mathfrak{g})$ and $\delta^{f_\mathfrak{g}} (f_V) $ are respectively given by
\begin{align*}
    (\delta  f_\mathfrak{g}) (x_1, \ldots, x_{n+1}) =~& \sum_{i=1}^{n+1} (-1)^{i+1} [\alpha^{n-1} (x_i), f_\mathfrak{g} (x_1, \ldots, \widehat{x_i}, \ldots, x_{n+1})] \\
    &+ \sum_{1 \leq i < j \leq n+1} (-1)^{i+j} f_\mathfrak{g} \big( [x_i, x_j], \alpha(x_1), \ldots, \widehat{\alpha (x_i)}, \ldots, \widehat{\alpha(x_j)}, \ldots, \alpha (x_{n+1})   \big),
\end{align*}
\begin{align*}
 ( \delta^{f_\mathfrak{g}} (f_V))   (x_1, \ldots, x_n, v) =~& \sum_{i=1}^n (-1)^{i+1} \rho (\alpha^{n-1} (x_i)) f_V (x_1, \ldots, \widehat{x_i}, \ldots, x_n, v) \\
 &+ (-1)^{n-1} \rho (f_\mathfrak{g} (x_1, \ldots, x_n)) (\beta^{n-1}(v)) \\
    &+ \sum_{1 \leq i < j \leq n} (-1)^{i+j} f_V \big(  [x_i, x_j], \ldots, \widehat{\alpha (x_i)}, \ldots, \widehat{\alpha (x_j)}, \ldots, v  \big) \\
    &+  \sum_{i=1}^n (-1)^i f_V (\alpha (x_1), \ldots, \widehat{\alpha(x_i)}, \ldots, \alpha (x_n), \rho (x_i) v),
\end{align*}
for $x_1, \ldots, x_n \in \mathfrak{g}$ and $v \in V$. Since $[\pi, \pi]_B = 0$, it follows that $(\delta_\mathrm{HLR})^2 = 0$. That is, $\{ C^\bullet_\mathrm{HLR} (\mathfrak{g}^{[~,~]}_\alpha, V^\rho_\beta), \delta_\mathrm{HLR} \}$ is a cochain complex. The corresponding cohomology is called the {\bf cohomology} of the Hom-LieRep pair $(\mathfrak{g}_\alpha^{[~,~]}, V_\beta^\rho)$, and denoted by $H^\bullet_\mathrm{HLR} (\mathfrak{g}^{[~,~]}_\alpha, V^\rho_\beta).$

\medskip

We are now in a position to define the cohomology of a Hom-Lie-Leibniz triple $(\mathfrak{g}_\alpha^{[~,~]}, V_\beta^\rho, T)$. We define the $n$-th cochain group $C^n_\mathrm{HLLT} (\mathfrak{g}_\alpha^{[~,~]}, V_\beta^\rho, T)$ by
\begin{align*}
    C^n_\mathrm{HLLT} (\mathfrak{g}_\alpha^{[~,~]}, V_\beta^\rho, T) = \begin{cases}
        0 & \text{ if } n = 0,\\
        \mathrm{Hom}(\mathfrak{g}_\alpha, \mathfrak{g}_\alpha) \oplus \mathrm{Hom} (V_\beta, V_\beta) & \text{ if } n =1,\\
        C^n_\mathrm{HLR} (\mathfrak{g}_\alpha, V_\beta) \oplus C^{n-1}_T (V_\beta, \mathfrak{g}_\alpha) & \text{ if } n \geq 2 \\
         \quad = \mathrm{Hom}(\wedge^n \mathfrak{g}_\alpha, \mathfrak{g}_\alpha) \oplus \mathrm{Hom}(\wedge^{n-1} \mathfrak{g}_\alpha \otimes V_\beta, V_\beta) \oplus \mathrm{Hom}(V_\beta^{\otimes n-1}, \mathfrak{g}_\alpha).
    \end{cases}
\end{align*}
Note that an element $(f_\mathfrak{g}, f_V) \in \mathrm{Hom} (\mathfrak{g}_\alpha, \mathfrak{g}_\alpha) \oplus \mathrm{Hom}(V_\beta, V_\beta) = C^1_\mathrm{HLLT} (\mathfrak{g}_\alpha^{[~,~]}, V_\beta^\rho, T)$ gives rise to an element $(s^{-1} (f_\mathfrak{g} + f_V), 0) \in (s^{-1} \mathfrak{q} \oplus \mathfrak{a})_{-1}$. Here we assume that $\mathfrak{a}_{-1} = 0$. On the other hand, for $n \geq 2$, the space $C^n_\mathrm{HLLT} (\mathfrak{g}_\alpha^{[~,~]}, V_\beta^\rho, T)$ is isomorphic to $(s^{-1} \mathfrak{q} \oplus \mathfrak{a})_{n-2}$ via
\begin{align*}
     \mathrm{Hom}(\wedge^n \mathfrak{g}_\alpha, \mathfrak{g}_\alpha) \oplus \mathrm{Hom}(\wedge^{n-1} \mathfrak{g}_\alpha \otimes V_\beta, V_\beta) \oplus \mathrm{Hom}(V_\beta^{\otimes n-1}, \mathfrak{g}_\alpha) \ni ~  
 & (f_\mathfrak{g}, f_V, P) \\
     \leftrightsquigarrow & (s^{-1} (f_\mathfrak{g} + f_V), P) \in (s^{-1} \mathfrak{q} \oplus \mathfrak{a})_{n-2}.
\end{align*}
We define a map $\delta_\mathrm{HLLT} : C^n_\mathrm{HLLT} (\mathfrak{g}_\alpha^{[~,~]}, V_\beta^\rho, T) \rightarrow C^{n+1}_\mathrm{HLLT} (\mathfrak{g}_\alpha^{[~,~]}, V_\beta^\rho, T)$ by
\begin{align*}
\delta_\mathrm{HLLT} ((f_\mathfrak{g}, f_V) ) =~& - l_1^{(s^{-1} \pi, T)} (s^{-1} (f_\mathfrak{g}+ f_V), 0), \text{ for } n = 1,\\
\delta_\mathrm{HLLT} ((f_\mathfrak{g}, f_V, P) ) =~& (-1)^{n-2} l_1^{(s^{-1} \pi, T)} (s^{-1} (f_\mathfrak{g}+ f_V), P), \text{ for } n \geq 2.  
\end{align*}
Since $\big(  s^{-1} \mathfrak{q} \oplus \mathfrak{a} , \{ l_k^{    (s^{-1} \pi, T)} \}_{k \geq 1} \big)$ is an $L_\infty$-algebra (namely the $L_\infty$-algebra twisted by the Maurer-Cartan element $\theta = (s^{-1} \pi, T)$), it follows that $(l_1^{(s^{-1} \pi, T)})^2 = 0$. Therefore, $\{ C^\bullet_\mathrm{HLLT} (\mathfrak{g}_\alpha^{[~,~]}, V_\beta^\rho, T), \delta_\mathrm{HLLT} \}$ is a cochain complex. The corresponding cohomology is called the {\bf cohomology} of the Hom-Lie-Leibniz triple $(\mathfrak{g}_\alpha^{[~,~]}, V_\beta^\rho, T)$. It is denoted by $H^\bullet_\mathrm{HLLT} (\mathfrak{g}_\alpha^{[~,~]}, V_\beta^\rho, T).$

The coboundary map $\delta_\mathrm{HLLT}$ is explicitly given by
\begin{align*}
    \delta_\mathrm{HLLT} ((f_\mathfrak{g},f_V, P)) = (\delta (f_\mathfrak{g}), \delta^{f_\mathfrak{g}} (f_V),  (-1)^n d_T (P) + \Omega_T(f_\mathfrak{g}, f_V)),
\end{align*}
for $(f_\mathfrak{g}, f_V, P) \in \mathrm{Hom}(\wedge^n \mathfrak{g}_\alpha, \mathfrak{g}_\alpha) \oplus \mathrm{Hom}(\wedge^{n-1} \mathfrak{g}_\alpha \otimes V_\beta, V_\beta) \oplus \mathrm{Hom}(V_\beta^{\otimes n-1}, \mathfrak{g}_\alpha) = C^n_\mathrm{HLLT} ( \mathfrak{g}_\alpha^{[~,~]}, V_\beta^\rho, T )$. Here $\Omega_T : \mathrm{Hom}(\wedge^n \mathfrak{g}_\alpha, \mathfrak{g}_\alpha) \oplus \mathrm{Hom}(\wedge^{n-1} \mathfrak{g}_\alpha \otimes V_\beta, V_\beta) \rightarrow C^n_T (V_\beta, \mathfrak{g}_\alpha)$ is the map given by
\begin{align*}
    \Omega_T (f_\mathfrak{g}, f_V) (v_1, \ldots, v_n) = (-1)^n \big(  f_\mathfrak{g} (Tv_1, \ldots, Tv_n) - Tf_V (Tv_1, \ldots, Tv_{n-1}, v_n) \big), \text{ for } v_1, \ldots, v_n \in V.
\end{align*}



The cohomology of the complex 
$\{ C^\bullet_\mathrm{HLLT} (\mathfrak{g}_\alpha^{[~,~]}, V_\beta^\rho, T) , \delta_\mathrm{HLLT} \}$
is called the {\bf cohomology} of the Hom-Lie-Leibniz triple $(\mathfrak{g}_\alpha^{[~,~]}, V_\beta^\rho, T)$ and denoted by $H^\bullet_\mathrm{HLLT} (\mathfrak{g}_\alpha^{[~,~]}, V_\beta^\rho, T).$

\medskip

Let $(\mathfrak{g}_\alpha^{[~,~]}, V_\beta^\rho, T)$ be a Hom-Lie-Leibniz triple. Then one may consider the following cohomologies: (i) the cohomology of the embedding tensor $T$, (ii) the cohomology of the Hom-LieRep pair $(\mathfrak{g}_\alpha^{[~,~]}, V_\beta^\rho)$, (iii) the cohomology of the whole Hom-Lie-Leibniz triple $(\mathfrak{g}_\alpha^{[~,~]}, V_\beta^\rho, T)$. In the following result, we relate these cohomology groups by a suitable long exact sequence.

\begin{thm}
    Let $(\mathfrak{g}_\alpha^{[~,~]}, V_\beta^\rho, T)$ be a Hom-Lie-Leibniz triple. Then there is a long exact sequence connecting various cohomology groups
    \begin{align*}
        \cdots \rightarrow H^{n-1}_T (V_\beta, \mathfrak{g}_\alpha) \rightarrow H^n_\mathrm{HLLT} (\mathfrak{g}_\alpha^{[~,~]}, V_\beta^\rho, T) \rightarrow H^n_\mathrm{HLR} (\mathfrak{g}_\alpha^{[~,~]}, V_\beta^\rho) \rightarrow H^{n}_T (V_\beta, \mathfrak{g}_\alpha) \rightarrow \cdots .
    \end{align*}
\end{thm}

\begin{proof}
Consider the cochain complexes

 (i) the complex $\{ C^\bullet_T (V_\beta, \mathfrak{g}_\alpha), d_T \}$ of the embedding tensor $T$, 
 
 (ii) the complex $\{ C^\bullet_\mathrm{HLR} (\mathfrak{g}^{[~,~]}, V_\beta^\rho), \delta_\mathrm{HLR} \}$ of the Hom-LieRep pair $(\mathfrak{g}^{[~,~]}, V_\beta^\rho)$,
 
 (iii) the complex $\{ C^\bullet_\mathrm{HLLT} (\mathfrak{g}^{[~,~]}, V_\beta^\rho, T), \delta_\mathrm{HLLT} \}$ of the Hom-Lie-Leibniz triple $ (\mathfrak{g}^{[~,~]}, V_\beta^\rho, T)$. \\
 Then there is a short exact sequence of cochain complexes
\begin{align*}
    0 \rightarrow C^{\bullet -1}_T (V_\beta, \mathfrak{g}_\alpha) \xrightarrow{i} \equalto{C^\bullet_\mathrm{HLLT}  (\mathfrak{g}^{[~,~]}, V_\beta^\rho, T)}{ C^\bullet_\mathrm{HLR}  (\mathfrak{g}^{[~,~]}, V_\beta^\rho) \oplus   C^{\bullet -1}_T (V_\beta, \mathfrak{g}_\alpha) } \xrightarrow{p} C^\bullet_\mathrm{HLR}  (\mathfrak{g}^{[~,~]}, V_\beta^\rho) \rightarrow 0.
\end{align*}
As a consequence, we get the desired result.
\end{proof}


\medskip

In the following, we introduce representations of a Hom-Lie-Leibniz triple. We also generalize the above cohomology of a Hom-Lie-Leibniz triple in the presence of an arbitrary representation. 

\begin{defn}
    Let $(\mathfrak{g}_\alpha^{[~,~]}, V_\beta^\rho, T)$ be a Hom-Lie-Leibniz triple. A {\bf representation} of $(\mathfrak{g}_\alpha^{[~,~]}, V_\beta^\rho, T)$ is given by a quadruple $(\mathfrak{h}^\varrho_\gamma, W_\eta^\vartheta, S, \Theta)$ consisting of two representations $\mathfrak{h}^\varrho_\gamma$ and $W_\eta^\vartheta$ of the Hom-Lie algebra $\mathfrak{g}_\alpha^{[~,~]}$, a linear map $S \in \mathrm{Hom} (W_\eta, \mathfrak{h}_\gamma)$ and a map $\Theta  : V \rightarrow \mathrm{Hom} (\mathfrak{h}, W)$ (called the pairing map) satisfying
    \begin{align*}
        \Theta ( \beta (v) ) (\gamma (h)) =~& \eta \big(   \Theta (v)(h) \big), \\
        \Theta (\rho(x) v) \circ \gamma =~& \vartheta (\alpha (x) ) \circ \Theta (v) - \Theta( \beta(v) ) \circ \varrho(x),\\
        \varrho (T (v)) \circ S =~& S \circ \vartheta (T(v)), \text{ for } x \in \mathfrak{g}, v \in V \text{ and } h \in \mathfrak{h}.
    \end{align*}
\end{defn}

Let $(\mathfrak{g}_\alpha^{[~,~]}, V_\beta^\rho, T)$ be a Hom-Lie-Leibniz triple. Then it can be seen as a representation $(\mathfrak{g}_\alpha^{[~,~]}, V_\beta^\rho, T, \Theta_\mathrm{ad})$, where the pairing map $\Theta_\mathrm{ad} : V \rightarrow \mathrm{Hom} (\mathfrak{g}, V)$ is given by
\begin{align*}
   \Theta_\mathrm{ad} (v) (x) = - \rho (x) v, \text{ for } v \in V, x \in \mathfrak{g}.
\end{align*}
This is called the adjoint representation of the Hom-Lie-Leibniz triple $(\mathfrak{g}_\alpha^{[~,~]}, V_\beta^\rho, T)$.

The usual semidirect product construction fits perfectly with Hom-Lie-Leibniz triples. More precisely, we have the following result.

\begin{prop}
    Let $(\mathfrak{g}_\alpha^{[~,~]}, V_\beta^\rho, T)$ be a Hom-Lie-Leibniz triple and $(\mathfrak{h}^\varrho_\gamma, W_\eta^\vartheta, S, \Theta)$ be a representation of it. Then the triple $\big(  (\mathfrak{g} \oplus \mathfrak{h})^{[~,~]_\ltimes}_{\alpha \oplus \gamma}  , (V \oplus W)_{\beta \oplus \eta}^{\rho_\ltimes} , T \oplus S \big)$ is a Hom-Lie-Leibniz triple, where
    \begin{align*}
[(x, h), (y, k) ]_\ltimes :=~& ([x,y], \varrho(x) k - \varrho(y) h),\\
\rho_\ltimes ((x,h)) (v, w) :=~& (\rho(x)v , \vartheta (x) w - \Theta (v) h),\\
( T \oplus S) ((v,w)) :=~& (Tv, Sw),
    \end{align*}
for $(x, h) , (y, k) \in \mathfrak{g} \oplus \mathfrak{h}$ and $(v, w) \in V \oplus W$. This is called the semidirect product.
\end{prop}

\begin{proof}
    The proof follows from a straightforward calculation.
\end{proof}

Let $(\mathfrak{g}_\alpha^{[~,~]}, V_\beta^\rho, T)$ be a Hom-Lie-Leibniz triple and $(\mathfrak{h}^\varrho_\gamma, W_\eta^\vartheta, S, \Theta)$ be a representation of it. For each $n \geq 0$, we define the space of $n$-cochains by
\begin{align*}
    C^n_\mathrm{HLLT} (\mathfrak{g}_\alpha^{[~,~]}, V_\beta^\rho, T ; \mathfrak{h}^\varrho_\gamma, W_\eta^\vartheta, S) = \begin{cases}
        0 \quad  \text{ if } n = 0, \\
        \mathrm{Hom} (\mathfrak{g}_\alpha, \mathfrak{h}_\gamma) \oplus \mathrm{Hom} (V_\beta, W_\eta) \quad  \text{ if } n =1,\\
        \mathrm{Hom} (\wedge^n \mathfrak{g}_\alpha, \mathfrak{h}_\gamma) \oplus \mathrm{Hom} (\wedge^{n-1} \mathfrak{g}_\alpha \otimes V_\beta, W_\eta) \oplus \mathrm{Hom} (V_\beta^{\otimes n-1} , \mathfrak{h}_\gamma) 
       ~~~ \text{if } n \geq 2.
    \end{cases}
\end{align*}
The coboundary map $\delta_\mathrm{HLLT} :  C^n_\mathrm{HLLT} (\mathfrak{g}_\alpha^{[~,~]}, V_\beta^\rho, T ; \mathfrak{h}^\varrho_\gamma, W_\eta^\vartheta, S) \rightarrow  C^{n+1}_\mathrm{HLLT} (\mathfrak{g}_\alpha^{[~,~]}, V_\beta^\rho, T ; \mathfrak{h}^\varrho_\gamma, W_\eta^\vartheta, S)$ is given by
\begin{align*}
    \delta_\mathrm{HLLT} ((f_\mathfrak{g}, f_V, P)) = \big(  \delta (f_\mathfrak{g}), \delta^{f_\mathfrak{g}} (f_V),   (-1)^n d_{T,S} (P) + \Omega_{T,S} (f_\mathfrak{g}, f_V) \big),
\end{align*}
for $(f_\mathfrak{g}, f_V, P) \in C^n_\mathrm{HLLT} (\mathfrak{g}_\alpha^{[~,~]}, V_\beta^\rho, T ; \mathfrak{h}^\varrho_\gamma, W_\eta^\vartheta, S)$.
Here the maps 
\begin{align*}
   & \delta : \mathrm{Hom} (\wedge^n \mathfrak{g}_\alpha , \mathfrak{h}_\gamma) \rightarrow \mathrm{Hom} (\wedge^{n+1} \mathfrak{g}_\alpha , \mathfrak{h}_\gamma),\\
   & \delta^{f_\mathfrak{g}} : \mathrm{Hom} (\wedge^{n-1} \mathfrak{g}_\alpha \otimes V_\beta, W_\eta) \rightarrow \mathrm{Hom} (\wedge^{n} \mathfrak{g}_\alpha \otimes V_\beta, W_\eta),\\
   & d_{T, S} : \mathrm{Hom} (V_\beta^{\otimes n} , \mathfrak{h}_\gamma) \rightarrow \mathrm{Hom} (V_\beta^{\otimes n+1} , \mathfrak{h}_\gamma)
\end{align*}
are respectively given by
\begin{align*}
    (\delta (f_\mathfrak{g})) (x_1, \ldots, x_{n+1}) =~& \sum_{i=1}^{n+1} (-1)^{i+1} \varrho(\alpha^{n-1} (x_i)) f_\mathfrak{g} (x_1, \ldots, \widehat{x_i}, \ldots, x_{n+1}) \\
   & + \sum_{1 \leq i < j \leq n+1} (-1)^{i+j} f_\mathfrak{g} \big( [x_i, x_j], \alpha(x_1), \ldots, \widehat{\alpha (x_i)}, \ldots, \widehat{\alpha(x_j)}, \ldots, \alpha (x_{n+1})   \big),
\end{align*}
\begin{align*}
    (\delta^{f_\mathfrak{g}} (f_V)) (x_1, \ldots, x_n, v) =~& \sum_{i=1}^n (-1)^{i+1} \vartheta (\alpha^{n-1} (x_i)) f_V (x_1, \ldots, \widehat{x_i}, \ldots, x_n, v) \\
    &+ (-1)^{n} \Theta(\beta^{n-1}(v)) f_\mathfrak{g} (x_1, \ldots, x_n)  \\
    &+ \sum_{1 \leq i < j \leq n} (-1)^{i+j} f_V \big(  [x_i, x_j], \ldots, \widehat{\alpha (x_i)}, \ldots, \widehat{\alpha (x_j)}, \ldots, v  \big) \\
    &+  \sum_{i=1}^n (-1)^i f_V (\alpha (x_1), \ldots, \widehat{\alpha(x_i)}, \ldots, \alpha (x_n), \rho (x_i) v),
\end{align*}
\begin{align*}
    d_{T,S} = \llbracket T+S, - \rrbracket, \qquad \qquad
\end{align*}
for $x_1, \ldots, x_{n+1} \in \mathfrak{g}$ and $v \in V$. Finally, for $v_1, \ldots, v_n \in V$,
\begin{align*}
    \big(   \Omega_{T,S} (f_\mathfrak{g}, f_V) \big) (v_1, \ldots, v_n ) = (-1)^n \big(  f_\mathfrak{g} (Tv_1, \ldots, Tv_n) - Sf_V (Tv_1, \ldots, Tv_{n-1}, v_n) \big).
\end{align*}

\begin{prop}
    With the above notations, we have $(\delta_\mathrm{HLLT})^2 = 0$. 
\end{prop}

\begin{proof}
Since $(\mathfrak{g}_\alpha^{[~,~]}, V_\beta^\rho, T)$ is a Hom-Lie-Leibniz triple and $(\mathfrak{h}^\varrho_\gamma, W_\eta^\vartheta, S, \Theta)$ is a representation, the semi-direct product
\begin{align*}
     \big(  (\mathfrak{g} \oplus \mathfrak{h})^{[~,~]_\ltimes}_{\alpha \oplus \gamma}  , (V \oplus W)_{\beta \oplus \eta}^{\rho_\ltimes} , T \oplus S \big)
\end{align*}
is a Hom-Lie-Leibniz triple. Thus, one may consider the cochain complex 
\begin{align*}
 \big\{  C^\bullet_\mathrm{HLLT}  \big(  (\mathfrak{g} \oplus \mathfrak{h})^{[~,~]_\ltimes}_{\alpha \oplus \gamma}  , (V \oplus W)_{\beta \oplus \eta}^{\rho_\ltimes} , T \oplus S \big), \delta_\mathrm{HLLT} \big\}
\end{align*}
of the above semidirect product Hom-Lie-Leibniz triple (with coefficients in the adjoint representation). Then it turns out that
\begin{align*}
    C^\bullet_\mathrm{HLLT} (    \mathfrak{g}_\alpha^{[~,~]}, V_\beta^\rho, T ; \mathfrak{h}^\varrho_\gamma, W_\eta^\vartheta, S   ) \subset C^\bullet_\mathrm{HLLT}  \big(  (\mathfrak{g} \oplus \mathfrak{h})^{[~,~]_\ltimes}_{\alpha \oplus \gamma}  , (V \oplus W)_{\beta \oplus \eta}^{\rho_\ltimes} , T \oplus S \big)
\end{align*}
is a subcomplex. Hence we get our required result.
\end{proof}

The cohomology of the complex $\{ C^\bullet_\mathrm{HLLT} (\mathfrak{g}_\alpha^{[~,~]}, V_\beta^\rho, T ; \mathfrak{h}^\varrho_\gamma, W_\eta^\vartheta, S), \delta_\mathrm{HLLT} \}$ is called the {\bf cohomology} of the Hom-Lie-Leibniz triple $(\mathfrak{g}_\alpha^{[~,~]}, V_\beta^\rho, T)$ with coefficients in the representation $(\mathfrak{h}^\varrho_\gamma, W_\eta^\vartheta, S, \Theta)$. We denote the corresponding cohomology group by $H^\bullet_\mathrm{HLLT} (\mathfrak{g}_\alpha^{[~,~]}, V_\beta^\rho, T ; \mathfrak{h}^\varrho_\gamma, W_\eta^\vartheta, S).$

One could use our cohomology (with coefficients in a representation) to study abelian extensions of a Hom-Lie-Leibniz triple. In the following, we consider infinitesimal deformations of a Hom-Lie-Leibniz triple and show that our cohomology governs such deformations. In future, we will study some other applications of cohomology.


\medskip

\subsection{Deformations of Hom-Lie-Leibniz triples} In this subsection, we study infinitesimal deformations of a Hom-Lie-Leibniz triple $(\mathfrak{g}^{[~,~]}_\alpha, V_\beta^\rho, T)$. We show that equivalence classes of such infinitesimal deformations are classified by the second cohomology group $H^2_\mathrm{HLLT} ( \mathfrak{g}^{[~,~]}_\alpha, V_\beta^\rho, T)$.

\begin{defn}
    An {\bf infinitesimal deformation} of a Hom-Lie-Leibniz triple $(\mathfrak{g}^{[~,~]}_\alpha, V_\beta^\rho, T)$ is a Hom-Lie-Leibniz triple of the form
    \begin{align*}
    (\mathfrak{g}^{[~,~] + \epsilon [~,~]_1}_\alpha, V_\beta^{\rho + \epsilon \rho_1 }, T+ \epsilon T_1),
    \end{align*}
    where $\epsilon$ is a parameter that satisfies $\epsilon^2 = 0$.
\end{defn}

Let $\mathfrak{g}[[t]]$ (resp. $V[[t]]$) be the space of all formal power series in $t$ with coefficients in $\mathfrak{g}$ (resp. $V$). We extend the linear map $\alpha: \mathfrak{g} \rightarrow \mathfrak{g}$ to a linear map (denoted by the same notation) $\alpha: \mathfrak{g}[[t]]/(t^2) \rightarrow \mathfrak{g}[[t]]/(t^2)$ by $t$-linearity. Similarly, the map $\beta: V \rightarrow V$ can be extended to a linear map (denoted by the same notation) $\beta: V[[t]]/(t^2) \rightarrow V[[t]]/(t^2)$. Then an infinitesimal deformation of the Hom-Lie-Leibniz triple $(\mathfrak{g}^{[~,~]}_\alpha, V_\beta^\rho, T)$ can be seen as a triple $([~,~]_1, \rho_1, T_1)$ in which
\begin{align*}
    [~,~]_1 \in \mathrm{Hom}(\wedge^2 \mathfrak{g}_\alpha, \mathfrak{g}_\alpha), \quad \rho_1 : \mathfrak{g}_\alpha \rightarrow \mathrm{End}(V_\beta) ~~ \text{ and } ~~ T_1 \in \mathrm{Hom}(V_\beta, \mathfrak{g}_\alpha)
\end{align*}
which makes $\big(  (\mathfrak{g}[[t]]/(t^2) )_\alpha^{[~,~] + t [~,~]_1}, (V[[t]]/(t^2))_\beta^{\rho + t \rho_1}, T+ t T_1  \big)$ into a Hom-Lie-Leibniz triple over the ring ${\bf k}[[t]]/(t^2).$

\medskip

It follows that $(\mathfrak{g}^{[~,~] + \epsilon [~,~]_1}_\alpha, V_\beta^{\rho + \epsilon \rho_1 }, T+ \epsilon T_1)$ is an infinitesimal deformation of the Hom-Lie-Leibniz triple if and only if the following conditions hold:
\begin{align}
    [\alpha (x) , [y, z]_1] + [\alpha (x) , [y,z]]_1 + [\alpha(y), [z,x]_1] + [\alpha (y) , [z,x]]_1 + [\alpha (z), [x, y]_1] + [\alpha (z), [x, y]]_1 = 0, \label{inf-def-id1}\\
    \rho (\alpha (x))\circ \rho_1 (y) + \rho_1(\alpha(x)) \circ \rho (y) - \rho (\alpha (y)) \circ \rho_1 (x) - \rho_1 (\alpha (y)) \circ \rho (x) = \rho_1 ([x, y])\circ \beta + \rho ([x,y]_1) \circ \beta,\\
    [T(u), T_1(v)] + [T_1 (u), T(v)] + [T(u), T(v)]_1 = T \big(  \rho (T_1 u) v \big) + T \big(  \rho_1 (Tu) v  \big) + T_1 \big(   \rho (Tu) v \big), \label{inf-def-id3}
\end{align}
for $x, y, z \in \mathfrak{g}$ and $u, v \in V$. Hence, it follows from (\ref{inf-def-id1})-(\ref{inf-def-id3}) that  $( [~,~]_1 , \rho_1, T_1  )$ defines a $2$-cocycle in the cochain complex of the Hom-Lie-Leibniz triple $(\mathfrak{g}_\alpha^{[~,~]}, V_\beta^\rho, T)$ with coefficients in the adjoint representation. In other words, $[ ([~,~]_1 , \rho_1, T_1)] \in H^2_\mathrm{HLLT} (\mathfrak{g}_\alpha^{[~,~]}, V_\beta^\rho, T)$.

\begin{defn}
    Two infinitesimal deformations 
    \begin{align*}
    (\mathfrak{g}^{[~,~] + \epsilon [~,~]_1}_\alpha, V_\beta^{\rho + \epsilon \rho_1 }, T+ \epsilon T_1) ~ \text{ and } ~ (\mathfrak{g}^{[~,~] + \epsilon [~,~]'_1}_\alpha, V_\beta^{\rho + \epsilon \rho'_1 }, T+ \epsilon T'_1)
    \end{align*}
    are said to be {\bf equivalent} if there are linear maps $N \in \mathrm{Hom}(\mathfrak{g}_\alpha, \mathfrak{g}_\alpha)$ and $S \in \mathrm{Hom}(V_\beta, V_\beta)$ 
    such that the pair of maps
    \begin{align*}
        \big( \mathrm{id}_\mathfrak{g} + \epsilon N , \mathrm{id}_V + \epsilon S   \big)
    \end{align*}
    defines a homomorphism from $(\mathfrak{g}^{[~,~] + \epsilon [~,~]_1}_\alpha, V_\beta^{\rho + \epsilon \rho_1 }, T+ \epsilon T_1)$ to $(\mathfrak{g}^{[~,~] + \epsilon [~,~]'_1}_\alpha, V_\beta^{\rho + \epsilon \rho'_1 }, T+ \epsilon T'_1)$.
\end{defn}

It follows that two infinitesimal deformations as above are equivalent if and only if the following conditions are held:
\begin{align*}
    [~,~]_1 - [~,~]_1' = \delta (N), \quad \rho_1 - \rho_1' = \delta^N (S) ~~~ \text{ and } ~~~ T_1 - T_1' = T \circ S - N \circ T.
\end{align*}
Hence we get that $([~,~]_1 , \rho_1, T_1) - ([~,~]_1' , \rho_1', T_1') = \delta_\mathrm{HLLT} ((N,S))$. In other words, the $2$-cocycles $([~,~]_1 , \rho_1, T_1)$ and $([~,~]'_1 , \rho'_1, T'_1)$ correspond to the same element in $H^2_\mathrm{HLLT} (\mathfrak{g}_\alpha^{ [~, ~]}, V_\beta^\rho, T)$.


\medskip

On the other hand, if we have a $2$-cocycle $(\mu_1 , \rho_1, T_1)$, then it is easy to verify that
\begin{align*}
(\mathfrak{g}^{[~,~] + \epsilon \mu_1}_\alpha, V_\beta^{\rho + \epsilon \rho_1 }, T+ \epsilon T_1)
\end{align*}
is an infinitesimal deformation of the Hom-Lie-Leibniz triple $(\mathfrak{g}^{[~,~]}_\alpha, V_\beta^{\rho }, T)$. Moreover, cohomologous $2$-cocycles correspond to equivalent infinitesimal deformations. As a consequence of all the results described above, we have the following classification result for infinitesimal deformations.

\begin{thm}
    Let $(\mathfrak{g}^{[~,~]}_\alpha, V_\beta^{\rho }, T)$ be a Hom-Lie-Leibniz triple. Then there is a bijective correspondence between the equivalence classes of infinitesimal deformations of $(\mathfrak{g}^{[~,~]}_\alpha, V_\beta^{\rho }, T)$ and the second cohomology group $H^2_\mathrm{HLLT} (\mathfrak{g}^{[~,~]}_\alpha, V_\beta^{\rho }, T).$
\end{thm}

\medskip


\section{Homotopy embedding tensors and $HLeib_\infty$-algebras}\label{sec7}
The notion of $HL_\infty$-algebra (also called a strongly homotopy Hom-Lie algebra) was introduced by Sheng and Chen \cite{sheng-hl} as the homotopy analogue of Hom-Lie algebra. In this section, we define the notion of a homotopy embedding tensor on a $HL_\infty$-algebra with respect to a representation. Next, we introduce the notion of a $HLeib_\infty$-algebra (strongly homotopy Hom-Leibniz algebra) and show that a homotopy embedding tensor induces a $HLeib_\infty$-algebra structure.

\begin{defn}\label{hl-defn}(\cite{sheng-hl}) A {\bf $HL_\infty$-algebra} is a triple $( \mathcal{G}, \{ l_k \}_{k \geq 1}, \alpha)$ consisting of a graded vector space $\mathcal{G} = \oplus_{i \in \mathbb{Z}} \mathcal{G}_i$ equipped with a collection $\{ l_k : \mathcal{G}^{\otimes k} \rightarrow \mathcal{G} \}_{k \geq 1}$ of degree $1$ graded symmetric linear maps and a degree $0$ linear map $\alpha : \mathcal{G} \rightarrow \mathcal{G}$ such that

- each $l_k$ is multiplicative, i.e. $l_k \big(  \alpha(x_1), \ldots, \alpha (x_k )  \big) = \alpha ( l_k (x_1, \ldots, x_k)),$ for all $k \geq 1$,

- higher Hom-Jacobi identities are hold, i.e. for all $n \geq 1$ and homogeneous elements $x_1, \ldots, x_n \in \mathcal{G}$,
    \begin{align*}
\sum_{i+j=n+1} \sum_{\sigma \in \mathbb{S}_{(i, n-i)}} \varepsilon (\sigma)~ l_j \big(   l_i ( x_{\sigma (1)}, \ldots, x_{\sigma (i)}), \alpha^{i-1} ( x_{\sigma (i+1)}), \ldots, \alpha^{i-1} (x_{\sigma (n)}) \big) = 0.
    \end{align*}
\end{defn}

Note that our definition of a $HL_\infty$-algebra differs slightly from the one given in \cite{sheng-hl}. However, the two notions are related by a degree shift. In Definition \ref{hl-defn}, if we take $\alpha = \mathrm{id}_\mathcal{G}$ (the identity map), one obtains the definition of an $L_\infty$-algebra (given in Definition \ref{l-inf-defn}).

\begin{defn}
    Let $( \mathcal{G}, \{ l_k \}_{k \geq 1}, \alpha)$ be a $HL_\infty$-algebra. A {\bf representation} of this $HL_\infty$-algebra is a triple $(\mathcal{V}, \{ \rho_k \}_{k \geq 1}, \beta)$ consisting of a graded vector space $\mathcal{V} = \oplus_{i \in \mathbb{Z}} \mathcal{V}_i$ with a collection $\{ \rho_k :   \mathcal{G}^{\otimes k-1} \otimes \mathcal{V} \rightarrow \mathcal{V}  \}_{k \geq 1}$ of degree $1$ graded linear maps and a degree $0$ linear map $\beta : \mathcal{V} \rightarrow \mathcal{V}$ such that the following conditions are hold:

        - each $\rho_k$ is graded symmetric on the inputs of $\mathcal{G}$, i.e.
        \begin{align*}
            \rho_k (x_{\sigma(1)}, \ldots, x_{ \sigma (k-1)}, v) = \varepsilon (\sigma) \rho_k (x_1, \ldots, x_{k-1}, v), \text{ for all } k \geq 1 \text{ and } \sigma \in \mathbb{S}_{k-1},
        \end{align*}
        
         - each $\rho_k$ is multiplicative, i.e. $\rho_k \big(  \alpha(x_1), \ldots, \alpha (x_{k-1}), \beta(v)  \big) = \beta \big(  \rho_k (x_1, \ldots, x_{k-1}, v)  \big)$, for all $k\geq 1$,
        
         - for each $n \geq 1$ and homogeneous elements $x_1, \ldots, x_{n-1} \in \mathcal{G}$, $x_n \in \mathcal{V}$,
        \begin{align*}
           & \sum_{i+j = n+1} \sum_{\substack{\sigma \in \mathbb{S}_{(i, n-i)} \\ \sigma (i)  = n} }\varepsilon(\sigma )  (-1)^{j-1 + (  i+ |x_{\sigma(1)}| + \cdots + |x_{\sigma (i)}|)( |x_{\sigma (i+1)}| + \cdots + |x_{\sigma (n)}|)} \\
          &  \qquad \qquad \qquad \qquad \qquad \qquad \rho_j \big( \alpha^{i-1} (x_{\sigma(i+1)}), \ldots, \alpha^{i-1} (x_{\sigma(n)}),  \rho_i (x_{\sigma(1)}, \ldots, x_{\sigma(i)})\big) \\
           & \quad +  \sum_{i+j = n+1} \sum_{\substack{\sigma \in \mathbb{S}_{(i, n-i)} \\ \sigma (n)  = n} }\varepsilon(\sigma ) \rho_j \big(  l_i (x_{\sigma(1)}, \ldots, x_{\sigma(i)}), \alpha^{i-1}( x_{\sigma(i+1)}), \ldots, \beta^{i-1}(x_{\sigma(n)})\big) = 0. 
        \end{align*}
\end{defn}

\medskip

\noindent Note that, if $\sigma \in \mathbb{S}_{(i, n-i)}$ is a shuffle, then we have either $\sigma (i) = n$ or $\sigma (n) = n$. Hence 
\begin{align*}
\sum_{\sigma \in \mathbb{S}_{(i, n-i)}} = \sum_{\substack{\sigma \in \mathbb{S}_{(i, n-i)} \\ \sigma (i)  = n} } + \sum_{\substack{\sigma \in \mathbb{S}_{(i, n-i)} \\ \sigma (n)  = n} }.
\end{align*}
Thus, by using the graded symmetry of the structure maps of a $HL_\infty$-algebra, one can easily show that any $HL_\infty$-algebra $(\mathcal{G}, \{ l_k \}_{k \geq 1}, \alpha)$ is a representation of itself, where $\rho_k = l_k$ for $k \geq 1$. This is called the adjoint representation.

We will now define the notion of a $HLeib_\infty$-algebra. Given a graded Hom-vector space, we also construct a graded Lie algebra whose Maurer-Cartan elements precisely correspond to $HLeib_\infty$-algebra structures. The graded Lie bracket here generalizes Balavoine's bracket in the graded context.

\begin{defn}
    A {\bf $HLeib_\infty$-algebra} is a triple $(\mathcal{H} , \{ \pi_k \}_{k \geq 1}, \alpha)$ consisting of a graded vector space $\mathcal{H} = \oplus_{i \in \mathbb{Z}} \mathcal{H}_i$ with a collection $\{ \pi_k : \mathcal{H}^{\otimes k} \rightarrow \mathcal{H} \}_{k \geq 1}$ of degree $1$ linear maps and a degree $0$ linear map $\alpha : \mathcal{H} \rightarrow \mathcal{H}$ satisfying the following conditions:

    - (multiplicativity) $\pi_k (\alpha (x_1), \ldots, \alpha (x_k) ) = \alpha \big(  \pi_k (x_1, \ldots, x_k) \big),$ for all $k \geq 1$,
    
 - (higher Hom-Leibniz identities) for all $n \geq 1$ and homogeneous elements $x_1, \ldots, x_n \in \mathcal{H}$,
 \begin{align*}
\sum_{k+l = n+1} \sum_{i=1}^{n-l+1} \sum_{\sigma \in \mathbb{S}_{(i-1, l-1)}} (-1)^{\gamma_i}~& \varepsilon(\sigma) ~ \pi_k \big(  \alpha^{l-1} (x_{\sigma(1)}), \ldots, \alpha^{l-1} (x_{\sigma(i-1)}), \\
& \pi_l ( x_{\sigma(i)}, \ldots, x_{\sigma (i+l-2)}, x_{i+l-1}), \alpha^{l-1} (x_{i+l}), \ldots, \alpha^{l-1}(x_n) \big) = 0,
 \end{align*}
 where $\gamma_i = |x_{\sigma(1)}| + \cdots + |x_{\sigma(i-1)}|.$
\end{defn}

A Hom-Leibniz algebra (see Definition \ref{hom-leib-defn}) can be regarded as a $HLeib_\infty$-algebra concentrated in degree $-1$. In \cite{wang} the authors introduced Hom-Leibniz $2$-algebras as a categorification of Hom-Leibniz algebras. A $HLeib_\infty$-algebra whose underlying graded vector space is concentrated in two degrees can be realized as a Hom-Leibniz $2$-algebra. In the above definition, if we take $\alpha = \mathrm{id}_\mathcal{H}$, then one obtains the notion of $Leib_\infty$-algebras \cite{poncin,uch}.

Given a $HL_\infty$-algebra and a representation of it, one can construct a $HLeib_\infty$-algebra. More precisely, we have the following result.

\begin{prop}
    Let $(\mathcal{G}, \{ l_k \}_{k \geq 1}, \alpha)$ be a $HL_\infty$-algebra and $(\mathcal{V}, \{ \rho_k \}_{k \geq 1}, \beta)$ be a representation of it. Then the triple $(\mathcal{G} \oplus \mathcal{V}, \{ l_k \boxplus \rho_k \}_{k \geq 1}, \alpha \oplus \beta)$ is a $HLeib_\infty$-algebra, where
    \begin{align}\label{lboxr}
        (l_k \boxplus \rho_k) \big(   (x_1, v_1), \ldots, (x_k, v_k) \big) = \big(  l_k (x_1, \ldots, x_k), \rho_k (x_1, \ldots, x_{k-1}, v_k) \big),
    \end{align}
    for $k \geq 1$ and $(x_1, v_1), \ldots, (x_k, v_k) \in \mathcal{G} \oplus \mathcal{V}$. This is called the hemi-semidirect product.
\end{prop}

The proof follows from a straightforward calculation. The converse of the above proposition also holds true. More precisely, let $\mathcal{G} = \oplus_{i \in \mathbb{Z}} \mathcal{G}_i$ be a graded vector space, $\alpha : \mathcal{G} \rightarrow \mathcal{G}$ be a degree $0$ linear map and $\{ l_k : \mathcal{G}^{\otimes k} \rightarrow \mathcal{G} \}_{k \geq 1}$ be a collection of degree $1$ symmetric multiplicative linear maps. Suppose $\mathcal{V} = \oplus_{i \in \mathbb{Z}} \mathcal{V}_i$  is another graded vector space with a degree $0$ linear map $\beta : \mathcal{V} \rightarrow \mathcal{V}$ and a collection $\{ \rho_k : \mathcal{G}^{\otimes k-1} \otimes \mathcal{V} \rightarrow \mathcal{V} \}_{k \geq 1}$ of degree $1$ (and symmetric on the inputs of $\mathcal{L}$) multiplicative linear maps. Then $(\mathcal{G}, \{ l_k \}_{k \geq 1}, \alpha)$ is a $HL_\infty$-algebra and $(\mathcal{V}, \{ \rho_k \}_{k \geq 1}, \beta)$ is a representation if and only if $(\mathcal{G} \oplus \mathcal{V}, \{ l_k \boxplus \rho_k \}_{k \geq 1}, \alpha \oplus \beta)$ is a $HLeib_\infty$-algebra.

Let $\mathcal{H} = \oplus_{i \in \mathbb{Z}} \mathcal{H}_i$ be a graded vector space and $\alpha: \mathcal{H} \rightarrow \mathcal{H}$ be a degree $0$ linear map (not necessarily equipped with any additional structure). For any $n \in \mathbb{Z}$, let $\mathrm{Hom}^n( \overline{T}(\mathcal{H})_\alpha, \mathcal{H}_\alpha)$ be the space of all degree $n$ linear maps from the graded vector space $\overline{T}(\mathcal{H})$ to the graded vector space $\mathcal{H}$ that are multiplicative. Thus, an element $\pi \in \mathrm{Hom}^n ( \overline{T}(\mathcal{H})_\alpha, \mathcal{H}_\alpha )$ is given by a sum $\pi = \sum_{k \geq 1} \pi_k$, where $\pi_k \in \mathrm{Hom}^n (\mathcal{H}_\alpha^{\otimes k}, \mathcal{H}_\alpha)$ for $k \geq 1$. Consider the graded vector space $\oplus_{n \in \mathbb{Z}} \mathrm{Hom}^n(\overline{T}(\mathcal{H})_\alpha, \mathcal{H}_\alpha)$. This graded vector space is equipped with a graded Lie bracket defined as follows. For $\pi = \sum_{k \geq 1} \pi_k \in \mathrm{Hom}^n( \overline{T}(\mathcal{H})_\alpha, \mathcal{H}_\alpha)$ and $\varpi = \sum_{l \geq 1} \varpi_l \in \mathrm{Hom}^m( \overline{T}(\mathcal{H})_\alpha, \mathcal{H}_\alpha)$, we define an element $[\pi, \varpi]_B \in \mathrm{Hom}^{n+m}( \overline{T}(\mathcal{H})_\alpha, \mathcal{H}_\alpha)$ by
\begin{align}\label{b-bracket-inf}
    [\pi, \varpi]_B := \sum_{p \geq 1} \sum_{k+l = p+1} \big( \pi_k \diamond \varpi_l - (-1)^{mn} \varpi_l \diamond \pi_k  \big),
\end{align}
where
\begin{align*}
    (\pi_k \diamond \varpi_l) (x_1, \ldots, x_p) = \sum_{i=1}^{p-l+1} \sum_{\sigma \in \mathbb{S}_{(i-1, l-1)}} (-1)^{ |x_{\sigma (1)}| + \cdots + |x_{\sigma (i-1)}| } \varepsilon (\sigma) \pi_k \big( \alpha^{l-1}(x_{\sigma (1)}), \ldots, \alpha^{l-1} (x_{\sigma (i-1)}), \\ \varpi_l (x_{\sigma (i)}, \ldots, x_{\sigma (i+l-2)}, x_{i+l-1}), \ldots, \alpha^{l-1} (x_p)  \big).
\end{align*}
Note that the bracket (\ref{b-bracket-inf}) is a generalization of Balavoine's bracket in the graded context. This graded Lie algebra characterizes $HLeib_\infty$-algebra structures as Maurer-Cartan elements. More precisely, we have the following result.

\begin{thm}
    Let $\mathcal{H} = \oplus_{i \in \mathbb{Z}} \mathcal{H}_i$ be a graded vector space and $\alpha : \mathcal{H} \rightarrow \mathcal{H}$ be a degree $0$ linear map. Then there is a one-to-one correspondence between $HLeib_\infty$-algebra structures on $\mathcal{H}$ and Maurer-Cartan elements in the graded Lie algebra $\big(  \oplus_{n \in \mathbb{Z}} \mathrm{Hom}^n  ( \overline{T}(\mathcal{H})_\alpha, \mathcal{H}_\alpha), [~,~]_B \big).$
\end{thm}

\begin{proof}
Note that an element $\pi = \sum_{k \geq 1} \pi_k \in \mathrm{Hom}^0 (\overline{T}(\mathcal{H})_\alpha, \mathcal{H}_\alpha)$ corresponds to a collection $\{ \pi_k : \mathcal{H}^{\otimes k} \rightarrow \mathcal{H} \}_{k \geq 1}$ of degree $1$ multiplicative linear maps. Then it follows \}from (\ref{b-bracket-inf}) that $[\pi, \pi]_B = 0$ if and only if $(\mathcal{H} , \{ \pi_k \}_{k \geq 1}, \alpha)$ is a $HLeib_\infty$-algebra. Hence the result follows.
\end{proof}

Given a Hom-Lie algebra and a representation of it, we constructed a graded Lie algebra whose Maurer-Cartan elements are precisely embedding tensors (see Section \ref{sec3}). Here we generalize this idea in the homotopy context. Let $(\mathcal{G}, \{ l_k \}_{k \geq 1}, \alpha)$ be a $HL_\infty$-algebra and $(\mathcal{V}, \{ \rho_k \}_{k \geq 1}, \beta)$ be a representation of it. Consider the graded Lie algebra
\begin{align*}
    \mathfrak{l} = \big(  \oplus_{n \in \mathbb{Z}} \mathrm{Hom}^n (\overline{T} (\mathcal{G} \oplus \mathcal{V})_{\alpha \oplus \beta}, (\mathcal{G} \oplus \mathcal{V})_{\alpha \oplus \beta}   ), [~,~]_B  \big).
\end{align*}
Then it is easy to verify that $\mathfrak{a} = \oplus_{n \in \mathbb{Z}} \mathrm{Hom}^n (\overline{T}(\mathcal{V})_\beta, \mathcal{G}_\alpha)$ is a graded Lie subalgebra of $\mathfrak{l}$. Let $P: \mathfrak{l} \rightarrow \mathfrak{a}$ be the projection onto the subspace $\mathfrak{a}$. Finally, we define an element $\Delta \in l_1 =  \mathrm{Hom}^1 (\overline{T} (\mathcal{G} \oplus \mathcal{V})_{\alpha \oplus \beta}, (\mathcal{G} \oplus \mathcal{V})_{\alpha \oplus \beta}   )$ by
\begin{align*}
    \Delta = \sum_{k=1}^\infty (l_k \boxplus \rho_k),
\end{align*}
where the collection $\{ l_k \boxplus \rho_k \}_{k \geq 1}$ is defined in (\ref{lboxr}).
 Then we have $[\Delta, \Delta]_B=0$ as the collection $\{ l_k \boxplus \rho_k \}_{k \geq 1}$ defines a $HLeib_\infty$-algebra structure. Moreover, the quadruple $(\mathfrak{l}, \mathfrak{a}, P, \Delta)$ is a $V$-data in the sense of Definition \ref{defn-v-d}. As a consequence of Theorem \ref{thm-voro}, we get an $L_\infty$-algebra structure $\{ \widetilde{l}_k \}_{k \geq 1}$ on the graded vector space $\mathfrak{a}$. In other words, $ \big(  \mathfrak{a} =  \oplus_{n \in \mathbb{Z}} \mathrm{Hom}^n (\overline{T}(\mathcal{V})_\beta, \mathcal{G}_\alpha ), \{ \widetilde{l}_k \}_{k \geq 1}  \big)$ is an $L_\infty$-algebra, where
\begin{align*}
    \widetilde{l}_k (a_1, \ldots, a_k) = P [ \cdots[[\Delta, a_1]_B, a_2]_B, \ldots, a_k]_B, \text{ for } k \geq 1 \text{ and } a_1, \ldots, a_k \in \mathfrak{a}.
\end{align*}

\medskip

\begin{defn}
    Let $(\mathcal{G}, \{ l_k \}_{k \geq 1}, \alpha)$ be a $HL_\infty$-algebra and $(\mathcal{V}, \{ \rho_k \}_{k \geq 1}, \beta)$ be a representation of it. A {\bf homotopy embedding tensor} on the $HL_\infty$-algebra with respect to the representation is simply a Maurer-Cartan element of the $L_\infty$-algebra $\big(  \oplus_{n \in \mathbb{Z}} \mathrm{Hom}^n (\overline{T}(\mathcal{V})_\beta, \mathcal{G}_\alpha ), \{ \widetilde{l}_k \}_{k \geq 1}  \big).$
\end{defn}

Thus a homotopy embedding tensor is given by an element $\Pi \in \mathrm{Hom}^0 (\overline{T}(\mathcal{V})_\beta, \mathcal{G}_\alpha )$ that satisfies
\begin{align*}
\sum_{k=1}^\infty \frac{1}{k!} \widetilde{l}_k (\underbrace{\Pi, \ldots, \Pi}_{k \text{ times}}) = 0.
     \end{align*}
     Equivalently, $\Pi$ must satisfy
     \begin{align*}
         \sum_{k=1}^\infty \frac{1}{k!} P [ \cdots [[\Delta, \Pi]_B, \Pi]_B, \ldots, \Pi]_B = 0, \quad \text{i.e. ~} ~ P(e^{[-, \Pi]_B} \Delta) = 0, \text{ where } \Delta = \sum_{k=1}^\infty (l_k \boxplus \rho_k).
     \end{align*}

\begin{prop}
    Let $(\mathcal{G}, \{ l_k \}_{k \geq 1}, \alpha)$ be a $HL_\infty$-algebra and $(\mathcal{V}, \{ \rho_k \}_{k \geq 1}, \beta)$ be a representation of it. Let $\Pi$ be a homotopy embedding tensor. Then the triple $(\mathcal{V}, \{ \pi_k \}_{k \geq 1}, \beta)$ is a $HLeib_\infty$-algebra, where
    \begin{align*}
        \pi_k (v_1, \ldots, v_k) = \big(  e^{[-, \Pi]_B} \sum_{k=1}^\infty (l_k \boxplus \rho_k) \big) (v_1, \ldots, v_k), \text{ for } k \geq 1.
    \end{align*}
\end{prop}

\begin{proof}
Since $[\Delta, \Delta]_B = 0$, it follows that
\begin{align*}
    [e^{[-, \Pi]_B} \Delta ~, e^{[-, \Pi]_B} \Delta]_B = e^{[-, \Pi]_B} [\Delta, \Delta]_B = 0.
\end{align*}
Hence $e^{[-, \Pi]_B} \Delta$ is a Maurer-Cartan element in the graded Lie algebra 
\begin{align*}
\mathfrak{l} = \big(  \oplus_{n \in \mathbb{Z}} \mathrm{Hom}^n (\overline{T} (\mathcal{G} \oplus \mathcal{V})_{\alpha \oplus \beta}, (\mathcal{G} \oplus \mathcal{V})_{\alpha \oplus \beta}   ), [~,~]_B  \big).
\end{align*}
Hence the collection $\{ \pi_k \}_{k \geq 1}$ defines a $HLeib_\infty$-algebra structure on $\mathcal{V}$, where
\begin{align*}
    \pi_k = \big( e^{[-, \Pi]_B} \Delta \big)|_{\mathcal{V}^{\otimes k}}, \text{ for } k \geq 1.
\end{align*}
This completes the proof.
\end{proof}

\thispagestyle{empty}



\medskip

\noindent {\bf Acknowledgements.}  A. Das would like to thank Indian Institute of Technology (IIT) Kharagpur for providing the beautiful academic environment where his part of the research has been carried out. A. Makhlouf  would like IFCPAR/CEFIPRA  (Grant No. 6201-C
2019-0071) for their support.

\medskip

\noindent {\bf Data Availability Statement.} Data sharing does not apply to this article as no new data were created or analyzed in this study.

\end{document}